\documentclass[article,reqno]{amsart}
\usepackage{amsbsy}
\usepackage{amsmath}
\usepackage{amstext}
\usepackage{amsthm}
\usepackage{amscd}
\usepackage{amssymb}
\usepackage{amsfonts}
\usepackage[all]{xy}
\usepackage{cases}
\usepackage{comment}
\usepackage[colorlinks=true, linkcolor=blue, citecolor=blue, urlcolor=blue]{hyperref}
\theoremstyle{plain}
\newtheorem{thm}{Theorem}[section]
\newtheorem{cor}[thm]{Corollary}

\newtheorem{lem}[thm]{Lemma}
\newtheorem{prop}[thm]{Proposition}

\theoremstyle{definition}

\theoremstyle{remark}
\newtheorem{rem}{Remark}[section]

\numberwithin{equation}{section}

\newcommand{\uple}[1]{\text{\boldmath${#1}$}}

\newcommand{\mods}[1]{\,(\mathrm{mod}\,{#1})}
\def\stacksum#1#2{{\stackrel{{\scriptstyle #1}}
{{\scriptstyle #2}}}}

\newcommand\Xip{\Xi_p}
\newcommand\etavalue{1/30}
\newcommand\Xipt{\Xi^t_p}
\newcommand\Xipw{\Xi^w_p}
\newcommand\vphi{\varphi}
\newcommand\eps{\varepsilon}

\newcommand\Gal{{\mathrm {Gal}}}
\newcommand\Tr{{\mathrm {Tr}}}

\newcommand\Hom{{\mathrm {Hom}}}

\newcommand{\ov}[1]{\overline{#1}}

\newcommand{\peter}[1]{\langle{#1}\rangle}
\newcommand\sumsum{\mathop{\sum\sum}\limits}

\newcommand{\lf}{\lambda_f}

\newcommand{\Qb}{{\mathbb Q}}
\newcommand{\Qq}{{\mathbb Q}}
\newcommand{\Cc}{{\mathbb C}}

\newcommand{\Zz}{{\mathbb Z}}
\newcommand{\Zpt}{{\mathbb Z}^\times_p}
\newcommand{\Qf}{{\mathbb Q}_f}

\newcommand{\Cb}{{\mathbb C}}

\newcommand{\Zp}{{\mathbb Z}_p}
\newcommand{\Qp}{{\mathbb Q}_p}

\newcommand{\bfalpha}{\uple{\alpha}}
\newcommand{\bfbeta}{\uple{\beta}}

\newcommand{\mcL}{{\mathcal L}}

\newcommand{\sym}{\mathrm{sym}}

\newcommand{\sump}[1]{{\sum_{#1}}^{\!\!\!\!\!\!*}}

\newcommand{\hide}[1]{{}}

\providecommand{\keywords}[1]{\textbf{\textit{Index terms---}} #1}

\usepackage{etoolbox}
\AtBeginEnvironment{array}{\setlength{\arraycolsep}{50pt}}
\begin{document}

\title[Hecke fields and $L$-values]{Generation of Hecke fields by  squares   of cyclotomic twists of modular $L$-values}
\author{ Valentin Blomer, Ashay Burungale, Philippe Michel and Jun-hwi Min}
\address{Mathematisches Institut, Endenicher Allee 60, 53115 Bonn, Germany} \email{blomer@math.uni-bonn.de}
\address{The University of Texas at Austin, 
Department of Mathematics, 
2515 Speedway, 
Austin, TX 78712} \email{ashayk@utexas.edu}
\address{EPFL-SB-MATH-TAN, Station 8, 
1015 Lausanne, Switzerland}\email{philippe.michel@epfl.ch}
\address{Ulsan National Institute of Science and Technology, Ulsan, Korea}\email{beliefonme159@unist.ac.kr}

\begin{abstract}
		Let $f$ be a non-CM elliptic newform without a quadratic inner twist, $p$ an odd prime and $\chi$ a Dirichlet character of $p$-power  order and sufficiently large $p$-power conductor. 
		We show that the compositum $\Qb_{f}(\chi)$ of the Hecke fields associated to $f$ and  $\chi$ is generated by the square of the absolute value of the corresponding central $L$-value $L^{\rm alg}(1/2, f \otimes \chi)$ over $\Qb(\mu_p)$. 
		The proof is based among other things on techniques, like sieve methods, used for the recent resolution of unipotent mixing conjecture by the first and third named authors. 
		\end{abstract}

\thanks{First author supported by ERC Advanced Grant  101054336 and Germany's Excellence Strategy grant EXC-2047/1 -- 390685813 as well as the German research Foundation Project-ID 491392403 -- TRR 358. Second author supported by NSF grant DMS 2302064. Third author supported by SNF Grant 200021L-231880 and 200021-197045.}

\maketitle


	\section{Introduction}		
		\subsection{Hecke fields}				A basic invariant associated to a motive $\mathcal{M}$ over a number field is the algebraic part $L^{\rm alg}(\mathcal{M},0)$ of a corresponding critical $L$-value. It lies in the field of definition of $\mathcal{M}$. 
	     Since the $L$-values are expected to be often non-zero as $\mathcal{M}$ varies in a family (besides the case of central $L$-values for self-dual families with root number $-1$), 
	     one may ask if $L^{\rm alg}(\mathcal{M},0)$ generates the field of definition of $\mathcal{M}$ generically. The aim of this paper is to consider a variant of this problem for   $|L^{\rm alg}(\mathcal{M},0)|^2$ when $\mathcal{M}$ is associated to an elliptic newform (cf.~Theorem~\ref{mainthm}). 
	     		
As a precursor and motivation to this question, we recall		
a seminal result of Luo--Ramakrishnan: elliptic newforms are determined uniquely by the associated central $L$-values with cyclotomic or quadratic twists (cf.~\cite[Theorem A, B]{luo1997determination}). The question whether Hecke fields are generated by algebraic part of these $L$-values or powers thereof can be seen as a refinement of this statement.  
	 Luo-Ramakrishnan \cite{luo1997determination} and Sun \cite{sun2019generation} have shown that modular $L$-values twisted by cyclotomic characters generate Hecke fields.
	 
	  In this paper,  we initiate variants of this problem and consider  the generation of Hecke fields by squares of the modular $L$-values. 

		\medskip 
Let us first recall the works of Luo--Ramakrishnan and Sun. Given a holomorphic cuspidal newform $f$  with level $R$ and trivial central character, e.g.\ the modular form associated with an elliptic curve $E/\Qq$, its  {\em Hecke field} is the (finite degree, totally real) number field generated by its Hecke eigenvalues 
		\begin{align*}
		\mathbb{Q}_f=\mathbb{Q}(a_{f}(n):\ n\geq 1).
		\end{align*}

			Given a Dirichlet character $\chi$ of modulus $q$ with root number  $\eps(\chi)=\chi(-1)\in\{\pm 1\}$ and Gauss sum $G({\chi})$ (relative to  the standard additive character $e_q(n)=\exp(\frac{2\pi i n}{q})$), we denote by $$\mathbb{Q}_f(\chi)=\Qq_f(\chi(n):\ n\geq 1)$$
		the number field obtained by adjoining the image of $\chi$ to $\mathbb{Q}_f$ and by $L(s,f\otimes \chi)$ the  twisted Hecke $L$-function, normalized (analytically) so that its functional equation is centered at $s=1/2$.
 
		 By the work of Shimura \cite{Sh1, Sh2}, there exist two complex periods $\Omega_{f}^\pm\in \mathbb{C}^\times$ (depending only on $f$)  such that for any $\chi$ we have 
		\begin{align*}
		L^{\rm alg}(1/2, f \otimes \chi):= 	L_f(\chi):=\frac{G(\overline{\chi})L(1/2,f\otimes \chi)}{\Omega_f^{\epsilon(\chi)}}\in \mathbb{Q}_f(\chi);
		\end{align*}
		We call the $L^{\rm alg}(1/2, f \otimes \chi)$ {\em arithmetically normalized $L$-values}.
		
Shimura also proved a reciprocity law: for any $\sigma\in \Hom_\mathbb{Q}(\overline{\mathbb{Q}},\mathbb{C})$, one has
		\begin{align}\label{shim reci}
			L_f(\chi)^\sigma=L_{f^\sigma}(\chi^\sigma)
		\end{align}
		where $f^\sigma$ is the modular form with Hecke eigenvalues $a_{f}(n)^\sigma=\sigma(a_{f}(n))$ and $\chi^\sigma(n):=\sigma(\chi(n))$.
		
		Given a prime $p\geq 3$  not dividing $R$, a classical and important problem is the investigation of the variation of the arithmetically normalized $L$-values $L_{f}(\chi)$ when $\chi$ varies over the set $\Xi_p$  of Dirichlet characters whose conductor $q$ is a $p$-power. We denote by  $\Xi^w_p$ the subset of (wild) Dirichlet characters whose conductor and order are both $p$-powers. In other words, we fix the tame part to be trivial and average only over the wild part, so that for each given conductor $p^h$, $h \geq 2$, we obtain precisely one Galois orbit $\Xi^w_{p}(h)$ of characters having order $p^{h-1}$, cf.\ \eqref{decom} below.  Note that for   $\chi \in \Xi^w_p$ we have $\eps(\chi)=+1$. 	A landmark result due to Rohrlich \cite{Rohr} is  that $L_f(\chi)\neq 0$  for all but finitely many $\chi\in\Xipw$. 
	
	A further important work is the afore-mentioned paper of Luo--Ramakrishnan \cite{luo1997determination} who proved that $f$ as well as its Hecke field can be recovered from the set of algebraic values $L_f(\chi)$ for $\chi\in\Xi_p^w$. 
	A consequence of their work concerns the field generated by the arithmetic {\em ratios} defined by
		\begin{align*}
			A_f(\chi,\chi_0)=L_f(\chi)L_f(\chi_0)^{-1}\in \overline{\mathbb{Q}}.
		\end{align*}
		for $\chi,\chi_0\in \Xi_p^w$ such that $L_f(\chi_0)\not=0$ (which exists as mentioned above). More precisely,  \cite[Theorem $C'$]{luo1997determination} states the following:

		\begin{thm} \label{Luo}
			Let $f$ be a newform of level $R$. Let $p\nmid 2R$ be a prime and let $\mu_{p^\infty}\subset \ov\Qq^\times$ be the group of $p$-power roots of unity. Then
			\begin{align*}
				\mathbb{Q}_f(\mu_{p^\infty})=\mathbb{Q}(\mu_{p^\infty},A_f(\chi,\chi_0):\chi\in \Xi^w_p).
			\end{align*}
		\end{thm}
		
		Recently, Sun \cite[Theorem D and Theorem 5.4]{sun2019generation} established a significant refinement of Theorem \ref{Luo}: 
		\begin{thm} \label{sun}
			Let $f$ be a newform of level $R$. Let $p\nmid R$ be a prime and let $\mu_{p}\subset\ov\Qq^\times$ be the group of $p$-th roots of unity. For all but finitely many $\chi \in \Xi^w_p$, one has
			\begin{equation*}
				\mathbb{Q}_f(\chi)=\Qq(\mu_p,L_f(\chi)).
			\end{equation*}
		
		\end{thm}

		\subsection{Main result}
		
	In this paper we consider a further refinement of this question and investigate when the following inclusion is an equality
			  $$
\Qq(\mu_p,|L_{f}(\chi)|^2)\subset \mathbb{Q}_{f}(\chi). 
$$
Recall that a holomorphic Hecke eigenform $f$ has a  {\em non-trivial quadratic inner twist} if there exist
$\sigma \in {\rm Aut}(\Cb)$ and a (non-trivial) quadratic character $\eta$ of conductor $D$ 	such that for any $p\nmid RD$ the Fourier coefficients $a_f$ of $f$ satisfy 
$$a_f(p)^{\sigma}=\eta(p)a_f(p).$$

Our main result is the following (see~Theorem \ref{main theorem later section} for a slightly more general result):
	\begin{thm}\label{mainthm}
	Let $f\in S^{{\rm new}}_{2\kappa}(R)$ be a non-CM newform of level $R$ and weight $2\kappa$ such that
	\begin{equation}
		\text{ $f$ has no non-trivial quadratic inner twist}.
\tag*{(Non-quad)} \label{non-quad}
	\end{equation}
 Let $p> 2\max(2,|\Hom_\mathbb{Q}(\mathbb{Q}_f,\mathbb{C})|)$ be an odd prime coprime to $R$. Then for all but finitely many $\chi\in \Xipw$, one has 
		$$
			\Qb_f(\chi)=\Qq(\mu_p,|L_f(\chi)|^2).
		$$
\end{thm}
\begin{rem} 
			If the level $R$ is square-free, then the condition \ref{non-quad} holds automatically, as shown by Ribet \cite{ribet1980twists}. 
\end{rem}	

\begin{rem} The assumption that $f$ is not a CM form is used at several places, e.g.\ when quoting the Sato--Tate law in Section \ref{sec52} and for certain properties of the Langlands parameters of $f$ for instance in Proposition \ref{quadtwist}. This is mostly for convenience, and recalling that a CM holomorphic modular form of weight $k \geq 2$ comes from a Hecke character, the situation in this case should   in fact be easier.  \end{rem}	


	\begin{cor}\label{cor14}
		In the setting of Theorem \ref{mainthm}, 
		assume further  that $$((p-1)p,[\Qb_{f}:\Qb])=1.$$ 
		Then for all but finitely many $\chi\in \Xipw$, we have
			\begin{align*}
				\mathbb{Q}(|L_f(\chi)|^2)\supseteq\mathbb{Q}_f.
			\end{align*}
	\end{cor}
We refer to the proof of Corollary \ref{Qf corollary} for details.  

\begin{rem}\label{remfnotg} Another possible variant is the following: given   two holomorphic Hecke eigenforms $f$ and $g$ of respective levels $R$ and $R'$ and trivial central character $p$ a prime not dividing $RR'$, what can be said of the set of $\chi\in \Xi^w_p$ for which 
		\begin{align*}
		\Qb_{f,g}(\chi)=\Qq(\mu_p,L_{f}(\chi)\cdot \overline{L_{g}(\chi)})\ \ \  (\hbox{here }\Qb_{f,g}=\Qb_{f}.\Qb_{g})? 
		\end{align*}
		 We hope to come back to the question in a later work.
	
\end{rem}

\subsection{About the proof} 
The present work is inspired by Sun's proof of Theorem~\ref{sun} 
which reduces the generation properties of Hecke fields to the non-vanishing of certain moments of the $L$-values $L_{f}(\chi)$ as $\chi$ varies in a {\em Galois orbit}. These moments are evaluated and proven to not vanish using analytic methods. 
The present paper was guided by the question of what happens if $L_{f}(\chi)$  is replaced with   $|L_{f}(\chi)|^2$.
 

To prove Theorem \ref{mainthm}, it suffices to establish the inclusions
\begin{align}
	\mathbb{Q}(\chi)\subseteq \Qq(\mu_p,|L_f(\chi)|^2)\label{firstinclusion}
\end{align}
(here $\mathbb{Q}(\chi):=\Qq(\chi(n),\ n\geq 1)$ denotes the cyclotomic field generated by the image of $\chi$) and
\begin{align}
\Qb_{f}\subseteq \mathbb{Q}(|L_f(\chi)|^2,\chi)\label{secondinclusion}
\end{align}
for almost all $\chi\in \Xipw$. 

  We approach the first  inclusion (\ref{firstinclusion}) by establishing the non-vanishing of some Galois average of   $|L_{f}(\chi)|^2$. Namely, we consider the following twisted averages 
\begin{align}
	\frac{1}{[\Qf(\chi):F_0]}\Tr_{\Qf(\chi)/F_0}(\chi(\ell)|L_f(\chi)|^2)\label{average in sketch}
\end{align}
for a subfield $F_0\subset \Qf$ and integers $\ell$ such that $\chi(\ell)$ is a primitive root of unity of order $p^{h-1}$. 

We show that there exists such an $\ell$, possibly dependent on $F_0$, such that \eqref{average in sketch} is non-zero as long as $\chi\in \Xipw$ has sufficiently large conductor.

Assume now that the first inclusion (\ref{firstinclusion}) does not hold for infinitely many $\chi\in \Xipw$. Set $K_{\chi}=\Qq(\mu_p,|L_f(\chi)|^2)$.
 Since $\mu_{p}\subseteq K_{\chi}$, we see that whenever $K_{\chi}\cap \mathbb{Q}(\chi)\neq \mathbb{Q}(\chi)$ we have 
 $$\Tr_{\mathbb{Q}(\chi)/K_{\chi}\cap \mathbb{Q}(\chi)}(\chi(\ell))=0=\Tr_{\Qf(\chi)/K_{\chi}}(\chi(\ell))$$
  (see also Remark \ref{orth remark}). By the transitivity of the trace $$\Tr_{\Qf(\chi)/\mathbb{Q}}=\Tr_{K_{\chi}/\mathbb{Q}}\circ\Tr_{\Qf(\chi)/K_{\chi}},$$
  we have
  $$\Tr_{\Qf(\chi)/\mathbb{Q}}(\chi(\ell)|L_f(\chi)|^2)=
  \Tr_{K_{\chi}/\mathbb{Q}}(|L_f(\chi)|^2\Tr_{\Qf(\chi)/K_{\chi}}(\chi(\ell)))=0$$ 
 for infinitely many $\chi\in \Xipw$, which is a contradiction.

The second inclusion (\ref{secondinclusion}) is also based on the non-vanishing of (\ref{average in sketch}), but now when  $F_0=\Qf$. 
It leads to a characterisation of $f$ in terms of the $L$-values: 
 for infinitely many $\chi$ the absolute values $|L_f(\chi)|^2$ determine the underlying newform $f$ (see Theorem \ref{determination L}). By applying Shimura reciprocity (\ref{shim reci}) and Theorem \ref{determination L} to the value $|L_f(\chi)|^2$, we conclude that $\Qf$ is fixed by  $ \text{Gal}(\overline{\mathbb{Q}}/\mathbb{Q}(|L_f(\chi)|^2,\chi))$. 

\medskip

Now we outline how the analysis of the second moment \eqref{average in sketch} proceeds. 
 
The starting point is the work of Mili\'cevi\'c and the first named author \cite{blomer2015second} who evaluated the second moment of twisted $L$-functions of the shape
$$\frac{1}{\vphi^*(q)}\sump{\chi\mods q}\chi(\ell)|L(1/2,f\otimes\chi)|^2$$
 (here "$*$" indicates summation  along primitive characters) for a suitable composite modulus $q$   and $\ell\geq 1$ is not too big compared to $q$ (see also \cite{BFKMM,KMS,blomer2015second} for the prime modulus case). 

When $q=p^k$ is a power of $p$, the approximate functional equation for the product 
 $|L(1/2,f\otimes\chi)|^2$ and summation over $\chi$ leads to evaluating sums of the shape
\begin{align}\label{introcongruencesums}
	\sum_{\substack{mn\ll p^{2h}\\ \ell m\equiv n (\text{mod }p^h)}}\frac{\lambda_{f}(m)\lambda_{f}(n)}{(mn)^{1/2}} 
\end{align}
for $h\leq k$ and where $\lambda_{f}(n)$ denotes the normalized $n$-th Hecke eigenvalue, i.e. $$\lambda_{f}(m)=a_f(m)/m^{(2\kappa-1)/2},$$ and $2\kappa$ is the weight of $f$. In this sum, the ``diagonal" contribution of the pairs $(m,n)$ satisfying $\ell m= n$ provides a main term which is a linear combination of products of the eigenvalues $\lf$ at prime powers dividing $\ell$ and which for $\ell=1$ has size \begin{equation}\label{introdiagonal}
	\gg_{f} \log q
\end{equation} by  the analytic properties of the Rankin-Selberg $L$-function $L(s,f\otimes f)$ near $s=1$. The remaining ``off-diagonal" terms constitute error terms. This comes from solving an unbalanced shifted convolution problem using spectral analysis on automorphic forms and large sieve type inequalities for Kloosterman sums.

When adjusting the above ideas into the realm of Hecke field generation, we have now to average over the {\em Galois conjugates} of $\chi$, instead of the primitive  characters with the same modulus $q$. For our application, $q=p^k$ is a prime power and we have to restrict the sum above to  the set $\Xipw(k)$ of (wild) characters of conductor $p^k$, i.e.\ having order $p^{k-1}$. This leads us into evaluating  ``restricted" moments of the shape
$$\frac{1}{\vphi^*(p^k)}\underset{\chi\mods {p^k}}{\left.\sum\right.^{\ast}}\chi^{p-1}(\ell)|L(1/2,f\otimes\chi^{p-1})|^2$$
where $*$ indicates an averaging over primitive characters. 
 Although $p$ is fixed and $k$ is growing, handling this kind of restriction presents significant analytic challenges, similar to the ones encountered (and resolved) in the work of Vatsal, Cornut and Cornut-Vatsal in the anticyclotomic setting  \cite{Vat,Cornut, CVDocumenta}. Via the approximate functional equation, we are reduced  to evaluate linear combinations of sums of the shape (we assume  $\ell=1$ for simplicity)
\begin{align*}
	\sum_{\substack{mn\ll p^{2k}\\m^{p-1}\equiv n^{p-1} (\text{mod }p^{h-h_0})}}\frac{\lambda_{f}(m)\lambda_{f}(n)}{(mn)^{1/2}}
\end{align*}
where $h_0$ is the largest non-negative integer such that $	\mathbb{Q}(\mu_{p^{h_0}})\subseteq \mathbb{Q}_f(\mu_p)$ and $h_0\leq h\leq k$. Assume (as this is the main case) that $h_0=1$ and $h=k$. By Hensel's Lemma ($p-1$ is coprime with $p$) the above sum decomposes into $p-1$ sums
\begin{align}\label{introcongruencesums2}
	\sum_{\substack{mn\ll p^{2k}\\\zeta m\equiv n (\text{mod }p^{k})}}\frac{\lambda_{f}(m)\lambda_{f}(n)}{(mn)^{1/2}}.
\end{align}
where $\zeta$ varies over the $(p-1)$-st roots of unity modulo $p^{k}$. The case $\zeta=1$ provides a main term satisfying \eqref{introdiagonal} and the other are expected to be error terms.

Let $\ell\in [1,p^k]$ be such that $\ell\equiv \zeta\mods{p^k}$ (a representative of the congruence class $\zeta$), then the sum \eqref{introcongruencesums2} is identical to  \eqref{introcongruencesums}. Unfortunately, the size of $\ell$ can be quite large (if $\zeta\not=\pm 1$ one has at least $| \ell | \geq p^{k/(p-1)}$, and in fact $\ell$ can be much larger) and  methods from the spectral theory of automorphic forms may not always apply.

Instead, we use a technique that was developed by the first and third named authors in \cite{blomer2023unipotent} to resolve the unipotent mixing conjecture for $q$ a prime and which was recently extended by Assing to general moduli \cite{Assing}. This technique makes crucial use of the multiplicativity of the Hecke eigenvalues $(\lambda_f(n))_{n\geq 1}$, Deligne's bound (coming from the holomorphy of $f$) and the fact that these Hecke eigenvalues satisfy the Sato--Tate law \cite{clozel2008automorphy}, \cite{taylor2008automorphy}, \cite{harris2010family} and \cite{newton2021symmetric}  (in fact, 
holomorphy of the first few symmetric power $L$-functions would be sufficient). The outcome is that whenever $\zeta\not=\pm 1$, the sums \eqref{introcongruencesums2} are bounded by
\begin{equation}
	\label{nondiagintro}
	\ll_{f}(\log q)^{1-\delta}
\end{equation} for some absolute $\delta\in (0,1)$. Comparing with \eqref{introdiagonal} we obtain an asymptotic formula and the non-vanishing for the Galois average.

 The evaluation of the main term of \eqref{average in sketch} when $F_0=\Qf$ is significantly different and combinatorially more involved than the approach of  Sun in \cite{sun2019generation}. In the present situation, the linear independence of the eigenforms in the Galois orbit of $f$ appearing in \eqref{average in sketch} is not applicable to establish the non-vanishing of our second moment. This is because in the main term of \eqref{average in sketch}, the Hecke eigenvalues of $f$ and their Galois conjugates appear (Dirichlet) polynomially intertwined (see Section \ref{local}). To get around this, we evaluate \eqref{average in sketch} with twisting factors $\chi(\ell^t)$ for $\ell$'s varying over a suitable set of primes of density $1/p$ (cf. Lemma \ref{v lemma}) and $t\geq 1$ some additional parameter bounded independently of $h$. To exhibit some non-vanishing in the resulting main terms, we proceed by contradiction and assume that all these main terms are vanishing. We view these as a recurrence sequence in $t$ and solving the recurrence relation, we obtain for each such $\ell$ an homogeneous linear system associated to a Vandermonde-like matrix admitting a nontrivial solution. Up to restricting further to a subset of primes $\ell$ with positive density, the vanishing of that determinant implies that two forms in the Galois orbit have matching Hecke eigenvalues along that subset of primes; by the work of Rajan \cite{Rajan}, these forms are quadratic twists of one another (cf.\ Theorem \ref{determination L}) contradicting the condition \ref{non-quad}  and therefore proving the non-vanishing of \eqref{average in sketch} (cf. Prop. \ref{nonvan}). 
 
\begin{rem} For the extension discussed  in Remark \ref{remfnotg} (for $f\not=g$), the present approach does not (yet) suffice: the main term appearing in \eqref{introdiagonal} satisfies only the weaker lower bound $\gg_{f,g}1$ and we would  need $\delta$ to be $>1$ in \eqref{nondiagintro} to compensate for its smaller size.  This case perhaps requires a key new idea. 
\end{rem}	
 
 \begin{rem} The methods of this paper extend {\em mutatis mutandis} to $p$-power characters of the shape $\varpi\chi$ where $\varpi$ is a {\em tame} character (i.e.\ of conductor $p$) and $\chi$ is varying over the wild characters: the same methods show that
 \end{rem}
 \begin{thm} With notations and assumptions as in Theorem \ref{mainthm}, let $\varpi\in \Xipt$ be a tame character. For all but finitely many $\chi\in \Xipw$, one has
\begin{align*}
	\Qb_{f}(\varpi\chi) = \Qq(\mu_p, |L_{f}(\varpi\chi)|^2).
\end{align*} 	
 \end{thm} 
 Varying $\varpi$ over the $p-1$ tame characters we obtain:
\begin{cor} Let the notations and assumptions be as in Theorem \ref{mainthm}. For all but finitely many Dirichlet characters, $\psi$ of $p$-power conductor, one has 
\begin{align*}
	\Qb_{f}(\psi) = \Qq(\mu_p, |L_{f}(\psi)|^2).
\end{align*}	
\end{cor}
To avoid adding the extra variable $\varpi$ to our notations we only write-up the case $\varpi=1$ and leave the general case to the interested reader. Even though the statement of this corollary involves all primitive characters of conductor $p^h$ for $h$ sufficiently large (and not only the wild characters), our reduction to the non-vanishing of suitable traces requires crucially the analysis of \emph{single Galois orbits}  which -- as described above -- is the source for significant analytic challenges.


\subsection{Outline of the paper}
	Section \ref{preliminary section} presents some preliminaries. We state useful facts such as approximate functional equation for a product of $L$-values and Galois average of Dirichlet characters.
	Section \ref{determination section} gives a determination result of newforms when given a series of linear equations involving Hecke eigenvalues of corresponding forms. This will be a key ingredient in Section \ref{generation of hecke field section} to show the non-vanishing of the Galois average. 
	In Section \ref{second moment section} and \ref{f=g error term section}, we prove the asymptotic estimates for the second twisted moment over a Galois orbit. In particular, we express the main terms in terms of the recurrence relation that arises from the multiplicativity of Hecke eigenvalues. 
The analytic techniques are modified adaptations of the works of the first named author and Mili\'cevi\'c \cite{blomer2015second} and of the first and third named authors \cite{blomer2023unipotent}.
	In Section \ref{generation of hecke field section}, we state our main results.  
\subsection*{Acknowledgements} Parts of this collaboration started during the AIM workshop ``Analytic, arithmetic, and geometric aspects of automorphic forms" that took place on the Caltech campus in January 2024 and which highlighted several of Dinakar Ramakrishnan's works, including the ones we are using here. We are grateful to the AIM and the department of Mathematics at Caltech for their hospitality and for providing an environment conductive to fruitful collaborations. We are also very grateful to Dinakar Ramakrishnan and Hae-Sang Sun for insightful discussions. Last but not least, special thanks are due to the anonymous referee for an extremely detailed reading of a first version of this manuscript; by catching   typos, asking for clarifications and by  identifying more serious errors, the referee has contributed to  improve the correctness and readability of this paper.

	\subsection*{Notations and Definitions}
\begin{itemize}
\item We denote by $\ov\Qq\subset \Cc$ an algebraic closure of $\Qq$ so that for any number field $K\subset \ov\Qq$ we have $\Hom_K(\ov\Qq,\Cc)=\Gal(\ov\Qq/K).$
\item Given $n\geq 1$ and $K$ a field we denote by $\mu_n(K)\subset K^\times$ the group of $n$-th roots of unity in $K$. When $K=\ov\Qq$ we will simply write $\mu_n$. 
	
	\item We use the notation $e(z)=\exp(2\pi iz)$.
	\item We denote by $d(n)=\sum_{d|n}1$ the number of divisors of an integer $n\geq 1$.

\item For an integer $r\geq 1$  and   an $L$-function $L(s)$ with  Euler product $L(s)=\prod_p L_p(s)$, we set $$L_r(s)=\prod_{p|r}L_p(s)\hbox{ and }L^{(r)}(s)=L(s)/L_r(s).$$

	\item Given an odd prime $p$  we choose an embedding of $\Qp\hookrightarrow \Cc$. We denote by $\Xip$ the group of Dirichlet characters of $p$-power conductor. This group is identified with the group of characters of $\Zpt$. From the decomposition
	\begin{equation}\label{decom}
	   \Zpt\subset \mu_{p-1}\times (1+p\Zp),
	   \end{equation} the group $\Xip$ decomposes as the product of the group of tame characters $\Xipt$ (of order dividing $p-1$ or equivalently of conductor $\mathfrak{f}(\chi)=1$ or $p$) and the group $\Xipw\subset \Xip$  of wild characters (of $p$-power order  or equivalently of conductor $\mathfrak{f}(\chi)=p^h$ for some $h\geq 2$). Conversely, for $h\geq 2$ denote by $\Xipw(h)$ the set of wild characters of conductor $p^h$. The set $\Xipw(h)$ forms a single Galois orbit and for any $\chi\in \Xipw(h)$ we have $\mathrm{Im}(\chi)=\mu_{p^{h-1}}$.

	\item We denote by $S^{\mathrm{new}}_{2\kappa}(R)$ the (finite) set of primitive newforms of weight $2\kappa$, level $R\geq 1$ and trivial central character. The $n$-th Fourier coefficient of such a newform $f\in S^{\mathrm{new}}_{2\kappa}(R)$ is noted $a_f(n)$ (with $a_f(1)=1$),  and its $n$-th (normalized) Hecke eigenvalue  is denoted 
	$$\lambda_{f}(n):=a_f(n)/n^{(2\kappa-1)/2},$$
	ie. the Fourier expansion of $f$ is
	$$f(z)=\sum_{n\geq 1}a_f(n)e(nz)=\sum_{n\geq 1}\lambda_f(n)n^{\frac{2\kappa-1}2}e(nz).$$
	This normalization for the Hecke eigenvalues leads to the functional equations for the $L(s,f\otimes \chi)$ mentioned in the beginning of this introduction and to the approximate functional equations in Proposition \ref{functionalequation}.

	\item For $\sigma \in \Hom_\mathbb{Q}(\overline{\mathbb{Q}},\mathbb{C})$  and $f\in S^{\mathrm{new}}_{2\kappa}(R)$ we denote by $f^\sigma$ the newform  whose Fourier expansion is
	 $$f^\sigma(z)=\sum_{n\geq 1}a_f(n)^\sigma e(nz)=\sum_{n\geq 1}\lambda_{f^\sigma}(n)n^{(2\kappa-1)/2} e(nz),$$ where $a_f(n)^\sigma=\sigma(a_{f}(n)).$
	
\end{itemize}

\section{Preliminaries}\label{preliminary section}

In this section, we compile a few facts  frequently used in the subsequent discussions.

\subsection{Basic properties of Hecke eigenvalues} 
Given $f\in S^{\mathrm{new}}_{2\kappa}(R)$, its  Hecke eigenvalues are multiplicative arithmetic functions satisfying the Hecke relations
\begin{equation}\label{mult2}
\begin{split}
	\lambda_f(m)\lambda_f(n)=\sum_{d|(m,n)}\chi_0(d)\lambda_{f}\bigg(\frac{mn}{d^2}\bigg),\\
	\lambda_{f}(mn)=\sum_{d|(m,n)}\chi_0(d)\mu(d)\lambda_f\bigg(\frac{m}{d}\bigg)\lambda_f\bigg(\frac{n}{d}\bigg)
	\end{split}
	\end{equation}
where $\chi_0$ is the trivial character modulo $R$. In particular, for $\ell$   prime and $t\geq 2$, one has 
\begin{align}\label{heckecyclotomic}
	\lambda_f(\ell^t)=\lambda_f(\ell)\lambda_f(\ell^{t-1})-\chi_0(\ell)\lambda_f(\ell^{t-2}).
\end{align}
or equivalently
$$\lf(\ell^t)=\sum_{r=0}^t\alpha^{r}_f(\ell)\beta^{t-r}_f(\ell)$$
where $\{\alpha_f(\ell),\beta_f(\ell)\}$ denote the {\em Langlands parameter} of $f$ at $\ell$, ie. the roots (possibly with multiplicity) of the quadratic polynomial $X^2-\lf(\ell)X+\chi_0(\ell)$.

By the work of Deligne \cite{WeilI} we have the Ramanujan-Petersson bound
$$|\alpha_f(\ell)|,|\beta_f(\ell)|\leq 1$$ so that
$$|\lf(\ell^t)|\leq t+1,\ |\lf(n)|\leq d(n)$$
	where $d(\bullet)$ is the divisor function. 

As mentioned in the introduction, the Fourier coefficients $(a_f(n))_{n\geq 1}$ are totally real algebraic integers and one has (see \cite[Corollary 12.4.5]{diamond1995modular}):
\begin{prop}
	Let $f(z)=\sum_{n}a_f(n)e(nz)\in S^{\mathrm{new}}_{2\kappa}(R)$ be a (normalized) Hecke newform. Then  $$\Zz_f:=\Zz[a_f(n),\ n\geq 1]$$
	is an order in a totally real number field $\Qq_f$. Moreover for any $\sigma\in \mathrm{Gal}(\overline{\mathbb{Q}}/\mathbb{Q})$, the holomorphic function $$f^\sigma(z)=\sum_{n\geq 1}a_f(n)^\sigma e(nz)$$
	 is also a normalised newform  contained in $S^{\mathrm{new}}_{2\kappa}(R)$.
\end{prop}

\subsection{Approximate functional equation and Voronoi summation}
A natural starting point for estimating moments of modular $L$-values is the approximate functional equation technique: it allows us  to express $L$-values by rapidly convergent sums.  We will use the following version from \cite[Proposition 2.18]{blomer2023second}.


\begin{prop}\label{functionalequation}  Let $f\in S_{2\kappa}^{{\rm new}}(R)$ and $g\in S_{2\kappa'}^{{\rm new}}(R')$ two normalized newforms with normalized Hecke eigenvalues $\lambda_f, \lambda_g$ as above. 
Let $\chi$ be a primitive Dirichlet character modulo $p^h$ with parity $\chi(-1)=\pm 1$ and $p$ coprime to $RR'$. 
Then 
	\begin{align*}
		L(1/2,f\otimes\chi)&\overline{L(1/2,g\otimes\chi)}=\sum_{m,n\geq 1}\frac{\lambda_f(m)\lambda_g(n)}{(mn)^{1/2}}\chi(m\overline{n})W_{f,g,\pm,\frac{1}{2}}\bigg(\frac{mn}{p^{2h}RR'}\bigg)\\
		&+\eps(f,g)\sum_{m,n\geq 1}\frac{\lambda_g(m)\lambda_f(n)}{(mn)^{1/2}}\chi(m\overline{n})W_{f,g,\pm,\frac{1}{2}}\bigg(\frac{mn}{p^{2h}RR'}\bigg),
	\end{align*}
	where $\ov n$ denote the multiplicative inverse of $n$ modulo $p^h$ and
\begin{equation}
	\label{Wdef}
	W_{f,g,\pm,s}(y)=\frac{1}{2\pi i}\int_{(2)}\frac{L_\infty(f,\pm,s+u)}{L_\infty(f,\pm,s)}\frac{L_\infty(g,\pm,\overline{s}+u)}{L_\infty(g,\pm,\overline{s})}e^{u^2}y^{-u}\frac{du}{u}.
\end{equation}
	and 
	\begin{align*}
		\eps(f,g)=\eps(f)\eps(g)\chi(R\overline{R'}).
	\end{align*}
\end{prop}Here, $\eps(f)$ is the root number of $f$ and $	L_\infty(f,\pm,s)$ is a product of shifted Gamma functions (see \cite[page 29]{blomer2023second}). 

Next we state the Voronoi summation formula for   holomorphic cusp forms $f$ (cf.\ e.g. \cite[Lemma 2.21, 2.23]{blomer2023second}) in the case of newforms.

\begin{lem}\label{Voronoi} Let $f\in  S_{2\kappa}^{{\rm new}}(R)$, $q \in \Bbb{N}$, $a\in \Bbb{Z}$ with $(aR, q) = 1$. Let $W$ be a smooth weight function with compact support in $(0, \infty)$ and $N > 1$. Then
$$\sum_{n} \lambda_f(n) W\Big(\frac{n}{N}\Big) e\Big(\frac{aq}{n}\Big) = \varepsilon(f) \frac{N}{q R^{1/2}} \sum_{n} \lambda_f(n) e\Big(- \frac{\overline{aR}n}{q}\Big) \widetilde{W}\Big( \frac{Nn}{q^2 R}\Big)$$
where the smooth function $y\mapsto \widetilde{W}(y)$ is a suitable integral transform (depending on $\kappa$) which satisfies 
$$\widetilde{W}(y) \ll _{A, f} (1 + y)^{-A}$$
for any $A > 0$. 
\end{lem}

A useful technique in analytic number theory is to express a long sum over integers into smooth localized sums using partitions of unity (see \cite[Lemme 2]{fouvry1985probleme}).
\begin{lem}\label{partition of unity}
	There exists a smooth and non-negative function $V$ with compact support $[1/2,2]$ and derivatives satisfying  $V^{(j)}(x)\ll_j 1$ for any $j\geq 0$ such that
	\begin{align*}
		\sum_{k\geq 0}V\bigg(\frac{x}{2^k}\bigg)=1
	\end{align*}
	for any $x\geq 1$.
\end{lem}


\subsection{\texorpdfstring{Rankin-Selberg $L$-function}{Rankin-Selberg L-function}}\label{rankinselbersection}

We recall the basic theory of Rankin-Selberg convolution. See \cite[\S 2.3]{blomer2023second} and \cite{kowalski2002rankin} for details. 

We define the ``naive" Rankin-Selberg convolution $L$-function as 
\begin{align*}
	L(s, f\otimes g)= \zeta(2s)\sum_{n\geq 1}\frac{\lambda_f(n)\lambda_{g}(n)}{n^s}=\prod_pL_p(s, f\otimes g)
\end{align*} 
which is absolutely convergent and admits an Euler product factorisation for $\Re e s>1$.
Closely related is the ``canonical" Rankin-Selberg $L$-function attached to the automorphic representations generated by $f$ and $g$ which we denote by
\begin{align*}
	L(s, \pi_f\otimes \pi_g)=\sum_{n\geq 1}\frac{\lambda_{\pi_f\otimes \pi_g}(n)}{n^s}=\prod_pL_p(s, \pi_f\otimes \pi_g).
\end{align*}
It is also absolutely convergent  for $\Re e s>1$ and has  an Euler product of degree $4$. In both cases, absolute convergence follow from the average bound 
\begin{equation}
	\label{eqRSbound}
	\sum_{n\leq x}|\lf(n)|^2+\sum_{n\leq x}|\lambda_g(n)|^2\ll_{f,g,\eps} x^{1+\eps}
\end{equation}	for any $\eps>0$ (itself a consequence of Rankin-Selberg theory in the case $f=g$) and the Cauchy-Schwarz inequality. Moreover, $L(s, \pi_f\otimes \pi_g)$ admits analytic continuation to $\Cc$ with a possible pole at $s=1$ of order at most $1$ and occurring if and only if $f=g$. It has  a functional equation of the shape
$$\Lambda(s, \pi_f\otimes \pi_g)=\Lambda(1-s, \pi_f\otimes \pi_g),$$
$$\Lambda(s, \pi_f\otimes \pi_g)=\eps(\pi_f\otimes \pi_g)q(\pi_f\otimes \pi_g)^{s/2}L_\infty(s, \pi_f\otimes \pi_g)L(s, \pi_f\otimes \pi_g)$$
where $\eps(\pi_f\otimes \pi_g)=\pm 1$, $ q(\pi_f\otimes \pi_g)\geq 1$  is an integer dividing $(RR')^2$ and $L_\infty(s, \pi_f\otimes \pi_g)$ is a product of Gamma factors.

The local $L$-factors of both $L$-functions coincide at almost all primes:
$$ L_p(s, f\otimes g)=L_p(s, \pi_f\otimes \pi_g)$$
whenever $(p, RR') = 1$, 
and hence their analytic properties are closely related; in particular we have 
\begin{prop}
	The function $L(s, f\otimes g)$ has a meromorphic continuation to $\mathbb{C}$. In the region $\Re e s > 0$ it has at most a simple   pole at $s=1$ which exists if and only if $f=g$. Moreover, if $f\neq g$ we have $L(1, f\otimes g)\neq 0$. 
\end{prop} 

We define the ``naive'' symmetric square $L$-function of $f$ (differing from the ``canonical'' $L$-function by Euler factors only at primes dividing $R$, specified in \cite[Section 2.3.3]{blomer2023second}) by
$$L(s, \text{sym}^2 f) \zeta(s) = L(s, f\otimes f),$$
so that $\text{res}_{s=1}L(s, f\otimes f) = L(1, \text{sym}^2 f)$. 


\subsection{Galois average of Dirichlet characters}
Let $p$ be an odd prime. 

Let $\chi\in \Xipw$,  $n$ be an integer coprime with $p$ and $F$ a number field. In this section, we evaluate the Galois average of $\chi(n)$ defined as
	\begin{align*}
	\chi_{{\rm av}}(n)=\frac{1}{[F(\chi):F]}\sum_{\sigma\in \text{Gal}(F(\chi)/F)}\chi(n)^\sigma.
\end{align*}

For this we need to introduce some notation. Given  a number field $F$ we define $h_0=h_0(F,p)\geq 1$ to be the largest  integer such that
\begin{align*}
	\mathbb{Q}(\mu_{p^{h_0}})\subseteq F(\mu_p).
\end{align*}
\begin{rem}
	Comparing degrees, we have
$$[\mathbb{Q}(\mu_{p^{h_0}}):\Qq]=\vphi(p^{h_0})=(p-1)p^{h_0-1}\leq [F(\mu_p):\Qq]\leq (p-1)[F:\Qq]$$
so that
$$h_0-1\leq \frac{\log [F:\Qq]}{\log p}\leq \frac{\log [F:\Qq]}{\log 2}.$$
In particular $h_0$ is bounded in terms of $F$ only.
\end{rem} 
Recall the decomposition
$$\Zp\simeq \mu_{p-1}\times(1+p\Zp).$$
In particular, any $n\in\Zz$ coprime with $p$ decomposes as the product
$$n=[n]\left<n\right>,\ [n]\in \mu_{p-1}(\Zp),\ \left<n\right>\in 1+p\Zp.$$
Given $h\geq 1$, we denote $[n]_h$ and $\peter{n}_h$ the congruence classes in $(\Zz/p^k\Zz)^\times$ such that
$$[n]\equiv [n]_h\mods{p^h},\ \peter{n}\equiv \peter{n}_h\mods{p^h}.$$
In the sequel we will write $\mu_{p-1}$ indifferently for $\mu_{p-1}$ or $\mu_{p-1}(\Zz/p^h\Zz)$.
 
See \cite[Lemma 3.1]{sun2019generation} for the proof of the following orthogonality relation. 
\begin{prop} \label{galoisav}
	Let $F$ be a number field as above and for $h>h_0$ let $\chi\in \Xipw(h)$. 
For $n$ an integer coprime with $p$ we have
	$$\chi_{{\rm av}}(n)=\begin{cases}
		\displaystyle \frac{1}{[F(\mu_p):F]}\sum_{\tau\in \Gal(F(\mu_p)/F)}\chi(n)^\tau&\hbox{ if }\left<n\right> \equiv 1\mods{p^{h-h_0}}\\
		0&\hbox{ if }\left<n\right> \not\equiv 1\mods{p^{h-h_0}}.
	\end{cases}$$
\end{prop}

\section{Determination of newforms}\label{determination section}

In this section we consider a tuple of newforms  whose Langlands parameters at a positive proportion of primes satisfy a suitable system of  equations and show that at least two of the forms are equal up to a quadratic twist. This will serve to show that the main term of a  second moment of twisted central $L$-values does not vanish, a crucial ingredient for Section \ref{generation of hecke field section}.  

The key ingredient is the following theorem which is essentially due to Rajan (see \cite[Cor.\ 1]{Rajan})

\begin{thm}\label{cor32} Let $f\in S_{2\kappa}^{{\rm new}}(R)$ and $g\in S_{2\kappa'}^{{\rm new}}(R')$ be non-CM newforms. Suppose that $\lf(\ell)=\lambda_g(\ell)$ for a positive proportion of primes $\ell$.
	Then there exists a real valued character $\chi_d$ of conductor $d$ dividing $RR'$ such that $f=g\otimes\chi_d$. 
	
In particular, if $R$ and $R'$ are square-free, we have $f=g$.
\end{thm}

\begin{rem} In a previous version of this paper we deduced Theorem \ref{cor32} from the following stronger result of Ramakrishnan \cite[page 30]{Ramappendix}. 
\begin{thm}\label{determinationRam} 
	Let $f\in S_{2\kappa}^{{\rm new}}(R)$ and $g\in S_{2\kappa'}^{{\rm new}}(R')$ be non-CM newforms such that
	\begin{align}\label{square equal}
		\lambda_f(\ell)^2=\lambda_g(\ell)^2 
	\end{align}
	for a positive proportion of primes $\ell$. Then there exists a real valued character $\chi_d$ of conductor $d \mid RR'$ such that $f=g\otimes\chi_d$. In particular,  for $(n, RR') = 1$ we have
	\begin{align*}
		\lambda_f(n)=\lambda_g(n)\chi_d(n)
	\end{align*} 
Moreover, if $R$ and $R'$ are square-free, we have $\chi_d=1$ and $f=g$. 
\end{thm}

However, as pointed out to us by the referee, in the present work, we need only Theorem \ref{cor32}; interestingly, Theorem 2 in Rajan's work \cite{Rajan} plays a crucial role in the proof  Theorem \ref{determinationRam}. We thank the referee for indicating the correct line of arguments. 	
\end{rem}

Let $(f_t)_{t\leq t_0}$, with $f_t\in S^{{\rm new}}_{2\kappa_t}(R_t)$, be a $t_0$-tuple of  non-CM newforms (not necessarily distinct). For  a prime $\ell$ coprime with $\prod_{t=1}^{t_0}R_i$, let $(c_{t,\ell},d_{t,\ell})_t\in \mathbb{C}^{2t_0}\setminus\{\mathbf{0}\}$ be a non-zero complex-valued vector and let $\{\alpha_{t}(\ell), \beta_{t}(\ell)\}$ be the Langlands parameter of $f_t$ at $\ell$, so that
$$\lambda_{f_t}(\ell)=\alpha_{t}(\ell)+\beta_{t}(\ell),\ \alpha_{t}(\ell)\beta_{t}(\ell)=1$$
for $1 \leq t \leq t_0$.  We have
\begin{prop}\label{quadtwist}
	Let $a\geq 0$ be an integer. Suppose that for a positive proportion of primes $\ell$ and for all  $s$ satisfying $a+1\leq s\leq a+2t_0$, we have 
	\begin{align*}
		\sum_{1\leq t\leq t_0}(c_{t,\ell}\alpha_{t}(\ell)^{s}+d_{t,\ell}\beta_{t}(\ell)^{s})=0.
	\end{align*}
Then there exist distinct $t_1\not= t_2$ and a real valued character $\chi_d$ of conductor $d$ dividing $R_{t_1}R_{t_2}$ such that $f_{t_1}=f_{t_2}\otimes \chi_d$. 	
	In particular, if the $R_{t}$'s are all square-free, there exist  $t_1\not= t_2$ such that $f_{t_1}=f_{t_2}$.
\end{prop}
\begin{proof} For any $1 \leq t \leq t_0$,  the proportion of primes $\ell$ such that $\alpha_{t}(\ell)=\beta_{t}(\ell)\, (=\pm 1)$ is zero (cf.\ \cite{serre1981quelques}),  so up to restricting the set of primes $\ell$, we may assume that for any $t\leq t_0$ and any $\ell$ in our set we have 
$\alpha_{t}(\ell)\not=\beta_{t}(\ell)$.

 For each $\ell$ in our set we have a system of linear equations
	\begin{align*}
		(\alpha_{t}(\ell)^{s},\beta_{t}(\ell)^{s})_{s,t}\begin{pmatrix}
			c_{t,\ell}\\d_{t,\ell}
		\end{pmatrix}_t=\begin{pmatrix}
		0
	\end{pmatrix}_{s}.
	\end{align*}
	Thus the determinant of the Vandermonde matrix $(\alpha_{t}(\ell)^{s},\beta_{t}(\ell)^{s})_{s,t}$ is zero.

	By the pigeonhole principle there exist $t\not=t'$ and a subset of primes of positive proportion such that for all $\ell$ in that subset, either of the following holds  
	$$\alpha_t(\ell)=\alpha_{t'}(\ell),\ \alpha_t(\ell)=\beta_{t'}(\ell)\hbox{ or } \beta_t(\ell)=\beta_{t'}(\ell)$$
	(we cannot have $\alpha_{t}(\ell)\not=\beta_{t}(\ell)$ by our initial discussion). In either of these cases we have for all such $\ell$'s that
$$\lambda_{f_{t_1}}(\ell)=\lambda_{f_{t_2}}(\ell).$$
 	We conclude by Theorem \ref{cor32}. 
\end{proof}

		\section{Second moment over Galois orbit}\label{second moment section}
	
	In this section, we provide estimates on Galois averages of $|L(1/2,f\otimes \chi)|^2$. This will be the key ingredient for the Theorem \ref{mainthm}. A thorough survey of the second moment theory can be found in \cite{blomer2023second}. 
	
	\subsection{The second moment result}
		
	\begin{thm}\label{general l1l2 sum} Let $f,g\in S_{2\kappa}^{{\rm new}}(R)$ be non-CM newforms with the same level and trivial nebentypus and such that $\eps(f)\eps(g)=1$. Let $F$ be a number field. 
	
	Given $h\geq 2$ and a prime $p$, set $q:=p^h$  and let $\chi\in\Xipw(h)$. Let   $l_1,l_2$ be positive integers such that $(l_1,l_2)=(l_1l_2,pR)=1$. We have	
		\begin{equation}\label{startsum} 
		\begin{split}
			\frac{(l_1l_2)^{1/2}}{[F(\chi):F]}&\sum_{\sigma\in \Gal(F(\chi)/F)}L(1/2,f\otimes\chi^{\sigma})\ov{L(1/2,g\otimes\chi^{\sigma})}\chi^{\sigma}(l_1/l_2)\\
			&= MT(f,g;q,l_1,l_2)+O_{F,f,g, l_1, l_2}((\log q)^{\frac{16}{3\pi}-1}). 
			\end{split}
			\end{equation}
Here 
	\begin{equation*}
	MT(f,g;q,l_1,l_2) =P_{f,g,l_1,l_2}(\log q)
	+P_{f,g,l_2,l_1}(\log q)
\end{equation*}
where $P_{f,g,l_1,l_2}(X)$ is a polynomial of degree at most $1$ whose coefficients depends at most on $f,g,l_1,l_2$.

More precisely  when $f\not=g$,
 this polynomial is a constant given by
 $$P_{f,g,l_1,l_2}(X)=\frac{L^{(pl_1l_2)}(1, f\otimes g)}{\zeta^{(pl_1l_2)}(2)}A_{l_1}(f,g;1)A_{l_2}(g,f;1)$$
with
\begin{align}
A_{l}(f,g;s)=\sum_{d|l^\infty}\frac{\lambda_{f}(ld)\lambda_{g}(d)}{d^s};\label{Alfsdef}
\end{align}
while for $f=g$,  $P_{f,f,l_1,l_2}(X)$ is a polynomial of degree $1$ with leading coefficient equal to
\begin{equation}
	\label{leadingcoeff}
\frac{L^{(pl_1l_2)}(1, \sym^2f)}{\zeta_{pl_1l_2}(1)\zeta^{(pl_1l_2)}(2)} A_{l_1}(f,f;1)A_{l_2}(f,f;1) 
\end{equation}
and whose constant term is bounded by $O_{F,f,l_1,l_2}(1)$.
	\end{thm}

Here is an overview of the arguments to come.
We analyse the average on the left hand side of \eqref{startsum} using the approximate functional equation of Proposition \ref{functionalequation}, switch summations and use Proposition \ref{galoisav}. We are led to evaluate a sum over two variables ($m$ and $n$, say) connected by a congruence equation $\ell_1m\equiv \pm \xi\ell_2n\mods q$ for $\xi$ varying over $(p-1)$-roots of unity modulo $q$. In \S  \ref{maintermsubsection} we isolate a main term (the contribution of the $m,n$'s satisfying  the equality $\ell_1m= \ell_2n$) which we evaluate up an error term bounded by a negative power of $q$. We then evaluate the complement, which is an error term. In \S \ref{errortermsubsection} and \S \ref{optimization subsection}, we consider the contribution of the terms with $\xi\equiv \pm 1\mods q$ using spectral analysis of automorphic forms (i.e.\ Kloostermania) with different bounds depending on the relative sizes of $m$ and $n$; there we obtain a power saving in $q$.

 The contribution of the terms with $\xi\not\equiv \pm 1\mods q$ is handled in Section \ref{f=g error term section} by a different set of  techniques, namely geometry of numbers, sieve methods and the Sato--Tate distribution properties of the Hecke eigenvalue $\lf$, $\lambda_g$ along the primes. The outcome is a power saving in $\log q$; this yields an asymptotic formula in \eqref{startsum} but only in the case $f=g$.

\begin{rem} When $f\not=g$,  \eqref{startsum} is not an asymptotic formula. This is because  $16/(3\pi) - 1>0$. Achieving a negative exponent, which we certainly expect, seems to require new ideas. Notice however that  \eqref{startsum} at least allows us to distinguish between the cases $f=g$ and $f\not=g$. 	
\end{rem}

\begin{rem} One can check that in the error terms, the dependence in the parameters $f,g,l_1,l_2,[F:\Qq]$  is at most polynomial;  we don't need this here.
\end{rem}

\subsection{The factors $A_l(f,g;s)$}\label{local}
Since the ``main term'' involves the factors $A_{l_1}(g,f;1)$, $A_{l_2}(f,g;1)$ it is necessary to compute them explicitly.
\begin{lem}\label{AflLemma} For $(l,R)=1$ we have
	$A_{l}(f,g;s)=
	\prod_{{\ell^t \parallel l}}A_{\ell^t}(f,g;s)$ 
	where
	$$A_{\ell^t}(f,g;s)=\sum_{r\geq 0}\frac{\lf(\ell^{t+r})\lambda_g(\ell^r)}{\ell^{rs}}.$$
	For $\alpha_f(\ell)\not =\beta_f(\ell)$ (equivalently for $\lf(\ell)\not=\pm 2$) we have for any integer $t\geq 0$,
\begin{align}
	A_{\ell^t}(f,g;s)&=\frac{L_{\ell}(s, f\otimes g)}{\zeta_{\ell}(2s)}(a_{f,g,\ell}(s)\alpha_f(\ell)^t+b_{f,g,\ell}(s)\beta_f(\ell)^t)\label{recurrence solutionfg}
\end{align}
where
\begin{equation*}
	a_{f,g,\ell}(s)=\frac{c_{f,g,\ell}(s)-\beta_f(\ell)}{\alpha_f(\ell)-\beta_f(\ell)},\ 
b_{f,g,\ell}(s)=\frac{\alpha_f(\ell)-c_{f,g,\ell}(s)}{\alpha_f(\ell)-\beta_f(\ell)}.\label{afgbfgdef}
\end{equation*}
and
\begin{equation*}
	\label{cfgdef}
	c_{f,g,\ell}(s):=\Big(\lf(\ell)-\frac{\lambda_g(\ell)}{\ell^s}\Big)\Big(1-\frac{1}{\ell^{2s}}\Big)^{-1}.
\end{equation*}

For $\alpha_f(\ell)=\beta_f(\ell)$ (equivalently for $\lf(\ell)=\pm 2$) we have
\begin{equation}\label{soldege}
	A_{\ell^t}(f,g;s)=\frac{L_{\ell}(s, f\otimes g)}{\zeta_{\ell}(2s)}(1-t+c_{f,g,\ell}(s)t).
\end{equation}

\end{lem}
\begin{proof}
By \eqref{heckecyclotomic} we have 
\begin{align*}
	A_{\ell^t}(f,g;s)=\lambda_f(\ell)A_{\ell^{t-1}}(f,g;s)-A_{\ell^{t-2}}(f,g;s)
\end{align*}
for a prime $(\ell, R) = 1$ and $t \geq 2$.  This is a recurrence relation with characteristic polynomial  $X^2-\lambda_f(\ell)X+1$ whose zeros are $\alpha_f(\ell),\beta_f(\ell)$. Moreover, we have
$$A_{\ell^0}(f,g;s)=\frac{L_\ell(s,f\otimes g)}{\zeta_\ell(2s)}$$
and for  $t=1$ we have by \eqref{heckecyclotomic}
$\lambda_f(\ell^{r+1})=\lambda_f(\ell)\lambda_f(\ell^{r})-\chi_0(\ell)\lambda_f(\ell^{r-1})$ so that
\begin{align*}
	A_{\ell}(f,g;s)&=\sum_{r\geq 0}\frac{\lf(\ell^{r+1})\lambda_g(\ell^r)}{\ell^{rs}}=\lf(\ell)+\lf(\ell)\sum_{r\geq 1}\frac{\lf(\ell^{r})\lambda_g(\ell^r)}{\ell^{rs}}-\sum_{r\geq 1}\frac{\lf(\ell^{r-1})\lambda_g(\ell^r)}{\ell^{rs}}\\
&=\lf(\ell)\frac{L_\ell(s,f\otimes g)}{\zeta_{\ell}(2s)}-\frac{1}{\ell^s}A_\ell(g,f;s).
\end{align*}

Switching the roles of $f$ and $g$ we have
$$A_{\ell}(g,f;s)=\lambda_g(\ell)\frac{L_\ell(s,f\otimes g)}{\zeta_{\ell}(2s)}-\frac{1}{\ell^s}A_\ell(f,g;s)$$
hence
$$A_{\ell}(f,g;s)=\lf(\ell)\frac{L_\ell(s,f\otimes g)}{\zeta_{\ell}(2s)}-\frac{1}{\ell^s}\bigl(\lambda_g(\ell)\frac{L_\ell(s,f\otimes g)}{\zeta_{\ell}(2s)}-\frac{1}{\ell^s}A_\ell(f,g;s)\bigr)$$
or in other terms
$$A_{\ell}(f,g;s)=\frac{L_\ell(s,f\otimes g)}{\zeta_{\ell}(2s)}(\lf(\ell)-\frac{\lambda_g(\ell)}{\ell^s})(1-\frac{1}{\ell^{2s}})^{-1}=\frac{L_\ell(s,f\otimes g)}{\zeta_{\ell}(2s)}c_{f,g,\ell}(s)$$

If $\alpha_f(\ell)\not=\beta_f(\ell)$ (i.e.\ $\lf(\ell)\not=\pm 2$), then $A_{\ell^t}(f;s)$ is equal to the expression \eqref{recurrence solution}
where $a_{f,\ell}(s)$ and $b_{f,\ell}(s)$ satisfy
$$
	a_{f,\ell}(s)+b_{f,\ell}(s)=1,\ \ 
	a_{f,\ell}(s)\alpha_f(\ell)+b_{f,\ell}(s)\beta_f(\ell)=c_{f,g,\ell}(s).
$$
If $\lf(\ell)=\pm 2$, the equation is degenerate and  the solutions are linear in $t$ and considering we cases $t=0,1$ we obtain to
\eqref{soldege}.

\end{proof}

\begin{rem}\label{remf=g} When $f=g$, we have
\begin{equation}
	\label{cffdef}c_{f,f,\ell}(s)=\frac{\lf(\ell)}{1+\ell^{-s}}
	\end{equation}
 To simplify notations, we will write  $A_l(f;s)$ instead of $A_l(f,f;s)$ and similarly $a_{f,\ell}(s)$, $b_{f,\ell}(s)$ and $c_{f,\ell}(s)$ for $a_{f,f,\ell}(s),\ b_{f,f,\ell}(s)$ and $c_{f,f,\ell}(s)$ respectively. With these notations, we have for $\lf(\ell)\not=\pm 2$ and $\ell\geq 0$ that
 \begin{equation}
	a_{f,\ell}(s)=\frac{c_{f,\ell}(s)-\beta_f(\ell)}{\alpha_f(\ell)-\beta_f(\ell)},\ 
b_{f,\ell}(s)=\frac{\alpha_f(\ell)-c_{f,\ell}(s)}{\alpha_f(\ell)-\beta_f(\ell)}\label{afbfdef}
\end{equation}
and 
\begin{align}
	A_{\ell^t}(f;s)&=\frac{L_{\ell}(s, f\otimes f)}{\zeta_{\ell}(2s)}(a_{f,\ell}(s)\alpha_f(\ell)^t+b_{f,\ell}(s)\beta_f(\ell)^t), \label{recurrence solution}
\end{align}
while for $\lf(\ell)=\pm 2$ we have
\begin{equation}\label{soldegef=g}
	A_{\ell^t}(f;s)=\frac{L_{\ell}(s, f\otimes f)}{\zeta_{\ell}(2s)}(1-t+c_{f,\ell}(s)t).
\end{equation}
\end{rem}

\subsection{Isolating the main term}\label{maintermsubsection}
	We begin by applying the approximate functional equation in Proposition \ref{functionalequation}. By averaging over Galois orbits of characters, we will see that the diagonal terms form the main term. 
	
	By Proposition \ref{functionalequation} and since $\chi(-1)=+1$, $\chi(R)\ov\chi(R)=1$ and $\eps(f)\eps(g)=1$, the term on the left-hand side of  (\ref{startsum}) equals 
	\begin{multline}
		(l_1l_2)^{1/2}\sum_{n,m}\frac{\lambda_f(m)\lambda_g(n)}{(nm)^{1/2}}W\bigg(\frac{nm}{q^2R^2}\bigg)\times\\\frac{1}{[F(\chi):F]}\sum_{\sigma\in \mathrm{Gal}(F(\chi)/F)}\bigl(\chi^\sigma(l_1m\overline{l_2n})+\chi^\sigma(l_2m\overline{l_1n})\bigr).
		\label{sums}
	\end{multline}
	where
	$$W(y):=W_{f,g,+,1/2}(y)$$
	is defined in \eqref{Wdef}.
	
		By Proposition \ref{galoisav}, the sum \eqref{sums} is the sum of two terms 
		\begin{multline}
(l_1l_2)^{1/2}
\sumsum_\stacksum{(nm,p)=1}{\peter{l_1m\overline{l_2n}}\equiv 1\, (\text{mod }p^{h-h_0})}\frac{\lambda_f(m)\lambda_g(n)}{(nm)^{1/2}}W\bigg(\frac{nm}{q^2R^2}\bigg)\times\\
	\frac{1}{[F(\mu_p):F]}	\sum_{\tau\in \text{Gal}(F(\mu_p)/F)}\chi^\tau(l_1m\overline{l_2n})\label{congruencesum}	
		\end{multline}
 and the same expression as 	\eqref{congruencesum} but with $l_1$ and $l_2$ exchanged.

		The diagonal terms  $l_1m=l_2n$  contribute the main term in \eqref{congruencesum} which we now evaluate. Since $(l_1,l_2)=1$ we can write $l_1m=l_2n=l_1l_2n'$   so that (writing $n$ for $n'$), this term contributes
		\begin{align}
			&\sum_{(n,p)=1}\frac{\lambda_f(l_2n)\lambda_g(l_1n)}{n}W\Big(\frac{l_1l_2n^2}{q^2R^2 }\Big)
			= \int_{(1)}D^{(p)}(f,g,2s+1;l_2,l_1)\Big( \frac{q^2 R^2}{l_1l_2}\Big)^s \widehat{W}(s)ds\label{cauchy integral}
		\end{align}
		where
		$$
			D^{(p)}(f,g,s;l_1,l_2)=\sum_{(n,p)=1}\frac{\lambda_{f}(nl_1)\lambda_{g}(nl_2)}{n^s} = A_{l_1}(f,g;s)A_{l_2}(g,f;s) \frac{L^{(pl_1l_2)}(s, f\otimes g)}{\zeta^{(pl_1l_2)}(2s)}
		$$
		where
	$A_{l}(f,g;s)$ is defined in \eqref{Alfsdef} and all expressions are 
  absolutely convergent for $\Re e(s)>1$. This follows since 
 $n\mapsto \lf(n)$ is multiplicative and $(l_1, l_2) = 1$. 
	
		Since $D^{(p)}(f,g,s;l_1,l_2)$ differs from $L(f\otimes g, s)$ by a finite Euler product, it has a holomorphic continuation to $\Re e(s)>1/2$ except for a simple pole at $s=1$ if $f=g$.
		Thus from \eqref{Wdef} we see that the integrand of \eqref{cauchy integral} has a pole of order at most 2 at $s=0$, so that, shifting the contour of integration in (\ref{cauchy integral}) to $(-1/4+\epsilon)$, we pick up a residue and the integral equals
		\begin{align*}
			P_{f,g,l_1,l_2}(\log q)+O_{f, g, l_1, l_2}(q^{-1/2+\epsilon})
		\end{align*}
		where
		$P_{f,g,l_1,l_2}(X)$ is the  polynomial of degree at most $1$ given in the statement of Theorem \ref{general l1l2 sum} with the indicated dominant term.  
		%


\subsection{The error term: first steps}\label{errortermsubsection}
In this section we initiate the evaluation of the  off-diagonal terms in \eqref{congruencesum}, i.e.\ the contribution of the $m,n$ satisfying the additional condition $l_1m\neq l_2 n$. 
We will show that they are bounded by $O_{h_0,f,g, l_1, l_2}((\log q)^{\frac{16}{3\pi} -1})$
and in particular  error terms when $f=g$. We recall that $h_0$ depends only on the number field $F$. 

The condition $\peter{u}=1\mods{p^{h-h_0}}$ is equivalent to $u=\xi(1+ap^{h-h_0})\mods{p^h}$ for some $\xi\in\mu_{p-1}(\Zz/p^{h}\Zz)$ and $a\in\Zz/p^{h_0}\Zz$, so we  have to evaluate
	\begin{multline}
			\frac{(l_1l_2)^{1/2}}{[F(\mu_{p}):F]}\sum_{\substack{\xi\in \mu_{p-1}\\0\leq a\leq p^{h_0}-1\\
			\tau\in \text{Gal}(F(\mu_p)/F)}}\chi^\tau(\xi(1+ap^{h-h_0}))\\
			\times\sum_{\substack{l_1m\neq l_2n\\l_1m\equiv \xi l_2n \, (\text{mod } p^{h-h_0})\\(mn,p)=1}}\frac{\lambda_{f}(m)\lambda_{g}(n)}{(nm)^{1/2}}W\bigg(\frac{nm}{q^2R^2}\bigg).\label{doublesumet}
	\end{multline}
	
	Since $p$ and $h_0$ are considered fixed, it suffices to show that for $\xi\in\mu_{p-1}$  we have
	\begin{equation}\label{ETgoal}
	ET(f, g; p^{h-h_0}, l_1, \xi l_2)\ll_{h_0,f,g, l_1, l_2}(\log p^h)^{\frac{16}{3\pi} - 1}
	\end{equation}
	where
\begin{equation}
	\label{ETdef} 
	ET(f, g; q, l_1, \xi l_2)	:= 	\sum_{\substack{l_1m\neq l_2n\\l_1m\equiv \xi l_2n \, (\text{mod } q)\\(mn,q)=1}}\frac{\lambda_{f}(m)\lambda_{g}(n)}{(nm)^{1/2}}W\bigg(\frac{nm}{q^2R^2}\bigg).
\end{equation}
	
 In this section we consider the case $\xi=\pm 1$; the general case will be discussed in \S \ref{f=g error term section}.
 
First we note that because of the congruence $\l_1m\equiv \pm l_2n \, (\text{mod } p^{h-h_0})$ and since $(\ell_1\ell_2,p)=1$, the condition $(mn,p)=1$ is equivalent to $(m,p)=1$; hence adding and subtracting the contribution of the $m$ divisible by $p$ and using the multiplicativity relations \eqref{mult2}
we are reduced to bounding sums of the shape
$$	\sum_{\substack{dl_1m\neq l_2n\\dl_1m\equiv \pm l_2n\, (\text{mod }p^{h-h_0})}}\frac{\lambda_{f}(m)\lambda_{g}(n)}{(mn)^{1/2}}W\bigg(\frac{dmn}{q^2R^2}\bigg)$$
where $d=1,p,p^2$. We will prove the following

\begin{lem}\label{SCPlem} Let $f,g\in S_{2\kappa}^{{\rm new}}(R)$, $p\geq 2$ a prime, $d\in\{1,p,p^2\}$, $l_1,l_2\geq 1$ coprime with $p$ and $h\geq h_0\geq 0$ two integers and let $q=p^h$. 

We have for any $\eps>0$ that 
  	\begin{equation}\label{SCPpm}
	\sum_{\substack{dl_1m\neq l_2n\\dl_1m\equiv \pm l_2n \, (\text{mod } p^{h-h_0})}}\frac{\lambda_{f}(m)\lambda_{g}(n)}{(mn)^{1/2}}W\bigg(\frac{dmn}{q^2R^2}\bigg)\ll_{f,g,p,h_0,\eps}(l_1l_2)^{3/2}q^{-\etavalue+\eps}
\end{equation}
as long as $l_1,l_2\leq q$.
\end{lem}
 This is a variation of
\cite[Thm 5.3]{blomer2023second} which was established with $q=p^{h}$ replaced by a prime, $d=1$ and with $\etavalue$ replaced by $1/144$.

\begin{rem}
	Rewriting the congruence relation $dl_1m\equiv \pm l_2n \, (\text{mod }p^{h-h_0})$ into the form
	$$dl_1m \mp l_2n=p^{h-h_0}k$$
	for $k$ is an integer satisfying
	$$1\leq |k|\ll \frac{dl_1M+l_2N}{p^{h-h_0}}$$
	($k$ is not zero because of the condition $dl_1m\not= l_2n$) and \eqref{SCPpm} is an instance of the {\em Shifted Convolution Problem}.
\end{rem}

	\begin{rem}
	We could have given a better exponent than $1/30$, but we have aimed for a relatively simple fraction as our point is  to furnish a positive  power saving in $q$.
	\end{rem}	

The proof is identical to that of \cite[Thm 5.3]{blomer2023second}: it is based on the work of the first named author and Mili\'cevi\'c \cite{blomer2015second}, the only difference being the following estimate for bilinear sums of Kloosterman sums. 

Given $r,l\geq 1$ be integers with $(l,r)=1$, $d\in \{1,p,p^2\}$, $M,N^*\geq 1$  and some complex numbers $\bfalpha=(\alpha_m)_{m\leq M},\bfbeta=(\beta_n)_{n\leq N^*}$, we define the bilinear sum
\begin{equation}
	\label{Babsum}
	B(\bfalpha,\bfbeta;r):=\sumsum_\stacksum{1\leq m\leq M}{1\leq n\leq N^*}\alpha_m\beta_n\frac{S(ldm,n;r)}{\sqrt{r}}
\end{equation}
where
$$S(m,n;r):=\sum_\stacksum{x\mods r}{(x,r)=1}e\Big(\frac{mx+n\ov x}{r}\Big)$$ the un-normalized Kloosterman sum modulo $r$.

\begin{lem}\label{BMlemma} Let the notation be as above and let $r=p^{h'}$. Suppose that for any $\eps>0$ and any $m\leq M,\ n\leq N^*$ we have
$$\alpha_m\ll_\eps m^{\eps},\ \beta_n\ll_\eps n^\eps,$$
then we have the upper bound
	\begin{equation}
	\label{BMbound}
	B(\bfalpha,\bfbeta;r)\ll_{\eps,p} (MN^*q)^\eps MN^*\min_{s\mid r}\Big(\Big(\frac{r}{N^*}\Big)^{1/2}+\Big(\frac{s}r\Big)^{1/4}+\Big(\frac{r}{M^2s}\Big)^{1/4}\Big).
\end{equation}
\end{lem}
\proof Given $s \mid r$,  by the Cauchy-Schwarz inequality we have
$$|B(\bfalpha,\bfbeta;r)|^2\ll_\eps \frac{M^{1+\eps}}{r}\sum_{m\leq M}\bigl|\sum_{n\leq N^*}\beta_n{S(dm,n;r)}\bigr|^2.$$
Following the proof of  \cite[Thm.\ 5]{blomer2015second} we obtain
$$|B(\bfalpha,\bfbeta;r)|^2\ll_{\eps,p} \frac{M^{1+\eps}}{r}(MN^*r)^\eps (M{N^*})^{2}\Big(\frac{s}{N^*}+(\frac{s}{r})^{1/2}+\frac{1}{M}(\frac{r}{s})^{1/2}\Big).$$
\qed

\begin{rem} Strictly speaking \cite[Thm.\ 5]{blomer2015second} provides a bound for bilinear sums of Kloosterman sums without the factor $d\in \{1,p,p^2\}$  in the argument of the Kloosterman sum and with the additional constraints that $m$ and $n$ vary over dyadic intervals (i.e.\ $m\in [\frac{1}{2}M,M]$, $n\in [\frac{1}{2}N^*,N^*]$) and $m$ is coprime with $p$, but since $p$ is fixed and the given upper bounds are increasing with $M$ and $N^*$, these extra modifications can easily be taken care of.
\end{rem}	

\subsection{Proof (sketch) of Lemma \ref{SCPlem}}
\label{optimization subsection}

	We now recall the argument from \cite[\S 5.7]{blomer2023second} adapted to the present situation using \eqref{BMbound} instead of \cite[Prop.\ 3.5]{blomer2023second} (for $q$ a prime).

	By a dyadic partition of unity the sum on the left hand side of \eqref{SCPpm} is bounded by $O((\log q)^2)$ sums of the shape
	$$ET(f,g;q,\pm;M,N):=\frac{1}{(MN)^{1/2}}\sum_{\substack{dl_1m\neq l_2n\\dl_1m\equiv \pm l_2n \, (\text{mod }p^{h-h_0})}}{\lambda_{f}(m)\lambda_{g}(n)}W_1\bigg(\frac{m}{M}\bigg)W_2\bigg(\frac{n}{N}\bigg)$$
	with $MN\leq q^{2+\eps}$ and $W_1(\bullet)$, $W_2(\bullet)$ are smooth compactly supported functions depending on a parameter of size $O_{f,g,p,h_0}((\log q)^2)$.

	We let $\eta=\etavalue$, and using with the notations of \cite[\S 5.7]{blomer2023second} we define
	$$M=q^\mu,\ N=q^\nu,\ \max(l_1,l_2)=q^{\lambda}.$$
	We have therefore
	  $$0\leq \lambda,\mu, \nu,\quad \lambda\leq 1/2, \quad  \mu+\nu\leq 2+\eps$$
	   and by symmetry we (may) assume that $$0\leq \mu\leq \nu,\quad \mu\leq 1+\eps.$$
	  We also set $$\mu^*=2-\mu,\ \nu^*=2-\nu$$
	  so that $$q^{\nu^*}=q^2/N\asymp N^*:=p^{h-h_0}R/N.$$
	 
	  Writing
	  $$|ET(f,g;q,\pm;M,N)|=q^\beta,$$
	 our aim is to prove that for $\lambda,\mu,\nu$ as above we have
	\begin{equation}
		\label{optimisationgoal}
		 \beta\leq \eps+\frac32\lambda-\eta,
	\end{equation}
	for any $\eps>0$. In the sequel we will omit  $\eps$ and will replace the inequalities $\leq$ with strict ones $<$.
	
	By the trivial bound in \cite[\S 5.4]{blomer2023second} (note that the Ramanujan conjecture is known in the holomorphic case) we have
	$$\beta<\lambda+
	\frac{1}2(\mu+\nu-2).$$
	
	Since otherwise \eqref{optimisationgoal} holds, we  assume that
	$$2-2\eta
	 \leq \mu+\nu\leq 2$$
	or equivalently
	$$-2\eta
	\leq \mu-\nu^*\leq 0;$$
	using that $\nu\leq 2$ we will simplify this condition to
	$$-2\eta 
	\leq \mu-\nu^*\leq 0.$$
	Using \cite[\S 5.5]{blomer2023second} we have
$$
\beta<
\frac32\lambda+\sup\Bigl(\frac{1}4(\nu-\mu-1),\frac{1}2(\nu-\mu-1),
-\frac{1}2 \Bigr),
$$
	and we have \eqref{optimisationgoal} unless
	\begin{equation*}
  1-4\eta\leq  \nu-\mu\quad\text{or equivalently}\quad \mu+\nu^*\leq 1+4\eta.	
\end{equation*}

Applying the Voronoi summation formula on the $n$ variable we transform the $n$-sum of length $N$ into a dual sum of length $\ll N^*$ which yields a bilinear sum of Kloosterman sums of the shape \eqref{Babsum} with $r=p^{h-h_0}$; using the individual ``Weil type'' bound  for Kloosterman sums (see \cite[Cor. 11.12]{IKbook} )
\begin{equation}
	\label{Weilbound}
	|S(m',n';r)|\leq d(r)(m',n',r)^{1/2}r^{1/2}
\end{equation}
we see that
$$\beta< \frac12(1+\mu-\nu)=\frac12(\mu+\nu^*-1),$$
which establishes \eqref{optimisationgoal} unless 
\begin{equation*}
1-2\eta \leq \mu+\nu^*\leq 1+4\eta.
\end{equation*}
Finally applying \eqref{BMbound} with $s=p^{h_1}=q^{\sigma}$ for some $0\leq h_1\leq h-h_0$, we obtain (recall that $p$ is fixed) 
$$\beta< \frac12(\mu+\nu^*-1)+\min_{0\leq \sigma <1}\max\Big(\frac{1}2(\sigma-\nu^*),\frac{1}4(\sigma-1),\frac{1}4(1-\sigma-2\mu)\Big).$$
We choose $\sigma$ such that $\frac{1}2(\sigma-\nu^*)=\frac{1}4(1-\sigma-2\mu)$, 
that is
$$\sigma = \frac{1 - 2\mu + 2\nu^*}{3},
$$
and so
$$\beta\leq \frac12(\mu+\nu^*-1)+\max\Big(\frac{1 - 2\mu - \nu^*}{6},\frac{-1 - \mu + \nu^*}{6}
\Big).$$
Since $\nu\leq 2$ we have
$$-2\eta
\leq \mu-\nu^*\leq 0\hbox{ and }1-2\eta \leq \mu+\nu^*\leq 1+4\eta$$
which implies that
\begin{align*}
	\frac{1 - 2\mu - \nu^*}{6}-\frac{-1 - \mu + \nu^*}{6}=\frac{2 - \mu - 2\nu^*}{6}&\geq \frac{1}{6}\Big(2-\frac{3}2(\mu+\nu^*)+\frac{\mu-\nu^*}{2}\Big)\\
	&\geq\frac{1}{12}\Big(\frac{1}{2}-\frac{7}\eta- \frac{7}{32}\Big)>0
\end{align*}
and
$$\max\Big(\frac{1 - 2\mu - \nu^*}{6},\frac{-1 - \mu + \nu^*}{6}
\Big)=\frac{1 - 2\mu - \nu^*}{6}$$
so that
\begin{align*}
	\beta&< \frac12(\mu+\nu^*-1)+\frac{1 - 2\mu - \nu^*}{6}=\frac{3(\mu+\nu^*)-(\mu-\nu^*)-4}{12}\\
	&< \frac{3(1+4\eta)+2(\eta+7/32)-4}{12}<0
\end{align*}
This concludes the sketch of the proof of \eqref{SCPpm}.\qed

\section{\texorpdfstring{Analysis of the $\xi\neq \pm 1$ case}{Analysis of kappa not equal to plus or minus 1 case}}\label{f=g error term section}

	In this section, we obtain a logarithmic upper bound for ${\rm ET}(f, g;p^{h-h_0},l_1,\xi l_2)$ where $\xi\in \mu_{p-1}(\Bbb{Z}/p^h\Bbb{Z})\setminus\{\pm 1\}$ as defined in \eqref{ETdef}.    We work slightly more generally and replace $p^{h-h_0}$ by a general modulus $q$. We note that for $\xi \not= 1$ the condition $l_1m \not= l_2n$ can be dropped.

		\begin{prop}\label{final-error} Let $q \in \Bbb{N}$, $(l_1l_2, q) = 1$  and $ \xi^d \equiv 1 \, (\text{\rm mod }q)$ for $d\geq 4$ where $\xi \not\equiv \pm 1\, (\text{\rm mod } q)$. Then
		  $${\rm ET}(f, g;q, l_1, l_2\xi) \ll_{\varepsilon, f, g} (\log q)^{\frac{16}{3\pi} - 1} + q^{-\frac{1}{2d} + \varepsilon} (l_1+l_2)$$
		  for any $\varepsilon > 0$.
	\end{prop}
	
	The rest of this section is devoted to the proof. 

\subsection{Geometry of numbers}
Before we begin with the analysis, we state the following well-known result from the geometry of numbers.
\begin{lem} \label{lattice}
Let $\Lambda  \subseteq \Bbb{Z}^2$ be a lattice of covolume $q$ and shortest nonzero vector of length $s$. 

{\rm (a)} Let $M, N \geq 1$ and $\mathcal{R} =  \{(x, y) \in \Bbb{R}^2 \mid M \leq x < 2M ,   N \leq y  < 2N\} $. Suppose that $\Lambda$ is given by 
\begin{equation}\label{form}
   \Lambda = \{(m, n) \in \Bbb{Z}^2 \mid a m \equiv b n \, (\text{{\rm mod }} q)\}
   \end{equation} with some $(ab, q) = 1$. 
Then
$$\sum_{(m, n) \in \Lambda \cap \mathcal{R}} 1 = \frac{MN}{q} + O\Big(1+\min\Big(\min(M, N) + \frac{M+N}{q}, \frac{M+N}{s}\Big)\Big).$$

{\rm (b)} Let $ T \geq 1$ and $\mathcal{R}\subseteq \Bbb{R}^2$ be contained in a ball of radius $T$ about the origin. Then
$$\sum_{0 \not= (m, n) \in \Lambda \cap \mathcal{R}} 1 \ll \frac{T^2}{q} + \frac{T}{s}.$$
\end{lem}
	
\begin{proof} (a) Since $\mathcal{R}$ is contained in a ball of radius $O(M+N)$, the asymptotic formula  $$\sum_{(m, n) \in \Lambda \cap \mathcal{R}} 1 = \frac{MN}{q}  + O\Big(1 + \frac{M+N}{s}\Big)$$ is a classical fact that holds without the assumption \eqref{form}, see e.g. \cite[Section 2]{Schmidt}.  Alternatively we may without loss of generality assume $M \leq N$ and compute
$$\sum_{(m, n) \in \Lambda \cap \mathcal{R}} 1 = \sum_{M \leq m < 2M} \sum_{\substack{N \leq n < 2N\\ n \equiv \bar{b} a m\, (\text{mod } q)}}1 =  \sum_{M \leq m < 2M} \Big(\frac{N}{q} + O(1)\Big) = \frac{MN}{q} + O\Big( M + 1 + \frac{N}{q}\Big).$$
This concludes the proof. 

(b)  There exists a $\Bbb{Z}$-basis $v_1, v_2$ of $\Lambda$ such that   $w = \alpha v_1 + \beta v_2 \in\Lambda$, $\alpha, \beta  \in \Bbb{Z}$, implies $\alpha	 \ll \| w \|/\| v_1\|$, $\beta \ll \| w \|/\| v_2\|$, see e.g.\  \cite[Lemma 5]{davenport}. By Hadamard's inequality we have $\| v_1\| \| v_2\| \geq q$ and clearly $\min(\| v_1\|, \| v_2 \|) \geq s$, so that 
$$\#\{(\alpha, \beta) \in \Bbb{Z}^2 \setminus \{0\} \mid  \alpha v_1 + \beta v_2  \in \mathcal{R}\} \ll \Big(1 + \frac{T}{\| v_1\|}\Big)\Big(1 + \frac{T}{\| v_2\|}\Big) - 1 \ll \frac{T^2}{q} + \frac{T}{s}$$
as desired. \end{proof}

We define the lattice  
\begin{equation}\label{deflambda}
\Lambda = \{(m,n)\in \mathbb{Z}^2\mid l_1m\equiv \xi l_2n\, (\text{mod }q)\},	
\end{equation}
and for later purposes we make the	following two observations: the covolume of $\Lambda$ is $q$ and the shortest nonzero vector of $\Lambda$ is at least of length
$$\frac{ q^{1/d} }{2\max(l_1, l_2)} .$$
Indeed, any nonzero vector $(m, n) \in \Lambda$ satisfies 
$(l_1m)^d \equiv (l_2n)^d \, (\text{mod } q)$. If $$\max(|n|, |m|) < \frac{q^{1/d} }{2\max(l_1, l_2)},$$
 then the congruence is an equality of integers, and hence $l_1m = \pm l_2n$, a contradiction to $\xi \not= \pm 1$.

\subsection{Sieves and Sato--Tate}\label{sec52}

	We quote the following application of sieve theory from  \cite[Section 3]{blomer2023unipotent}.

	For $X, Y \gg 1$ and $q \in \mathbb{N}$ we denote by $\mathcal{C}_q(X, Y)$ the set of  subsets   $\mathcal{S} \subseteq \mathbb{N}^2$ contained in a ball of radius $X^{100}$ about the origin and satisfying  
\begin{equation}\label{sieve-cond}
\sum_{\substack{ (n_1,n_2) \in \mathcal{S}\\ d_1 \mid n_1, d_2 \mid n_2}} 1 = \frac{X}{d_1d_2} + O(Y)
\end{equation}
for all $(d_1d_2, q) = 1$. 

\begin{lem}\label{thm-sieve} Let $\mathcal{S} \in \mathcal{C}_q(X, Y)$. 
Let $\lambda_1, \lambda_2$ be two non-negative multiplicative functions satisfying $\lambda_j(n)  \leq \tau_k(n)$ for some   $k \in \mathbb{N}$. Fix $0 < \gamma < 1/2$ and let $ z = 2+X^{\gamma}$. Then for any $\varepsilon > 0$ we have 
$$\sum_{\substack{ (m, n) \in \mathcal{S}\\ (mn, q) = 1}} \lambda_1(m) \lambda_2(n)  \ll_{k,  \gamma,\varepsilon} \frac{X}{(\log z)^2}
\exp\Big( \sum_{\substack{\ell \leq z\\\ell \text{ {\rm     prime}}}} \frac{\lambda_1(\ell) + \lambda_2(\ell)}{\ell} \Big) 
+  \frac{X^{1+\varepsilon}}{ z^{1/4}} + X^{\varepsilon} Y z^3.$$

\end{lem}

We estimate the sum over $\ell$ in the preceding formula in the case when $\lambda_1 =| \lambda_f|$, $\lambda_2 = |\lambda_g|$ using the Sato--Tate conjecture which is a now theorem  for holomorphic non-CM newforms $f, g$ (see \cite{clozel2008automorphy}, \cite{taylor2008automorphy}, \cite{harris2010family} as well as \cite{newton2021symmetric}). By \cite[Thm.\ 1.1]{Thorner}  we have for  $x \rightarrow \infty$ that 
\begin{align*}
			\frac{1}{\pi(x)}\sum_\stacksum{\ell \leq x}{\ell\text{ prime} }|\lambda_{f}(\ell)|&= \frac{2}{\pi}\int_0^\pi 2|\cos t| \sin^2 t \, dt+O_f\Big(\frac{\log\log x}{(\log x)^{1/2}}\Big)\\
			&=\frac{8}{3\pi}+O_f\Big(\frac{\log\log x}{(\log x)^{1/2}}\Big).
		\end{align*}
		By partial summation this gives
		  		\begin{equation}\label{sato}
			\sum_{\substack{\ell\leq z \\ \ell \text{ prime}}}\frac{|\lambda_{f}(\ell)| +|\lambda_{g}(\ell)| }{\ell} =  \frac{16}{3\pi}\log\log z+O_f(1).
		\end{equation}

\subsection{The key bound}

We now state 	the  key bound. We recall the definition \eqref{deflambda} of $\Lambda$. 
	
	
\begin{prop}\label{bound1} Let $N, M \geq 1$, $q\in \Bbb{N}$, $\xi^d \equiv 1 \, (\text{\rm mod } q)$ for some $d \geq 4$, $\xi \not\equiv \pm 1\, (\text{\rm mod } q)$, $(l_1l_2, q) = 1$, $\varepsilon > 0$, $W_1, W_2$ two fixed smooth test functions with support in $[1, 2]$. 

{\rm (a)} Suppose $MN \geq q$. Then
\begin{displaymath}
\begin{split}
 &\sum_{ \substack{(m, n) \in \Lambda\\ (mn, q) = 1}} \lambda_f(m) \lambda_g(n)W_1\Big(\frac{m}{M}\Big)W_2\Big(\frac{n}{N}\Big) \\
 &\ll_{f, g,  W_1, W_2,  \varepsilon} \frac{MN}{q} \Big(\log \Big(2 + \frac{MN}{q}\Big)\Big)^{\frac{16}{3\pi} -2}  +   
 (MNq)^{\varepsilon}\Big(\sqrt{\frac{ MN (l_1 + l_2)}{q^{1/d}}}\Big) .
 \end{split}
 \end{displaymath}
 
 {\rm (b)} In general, we have the (coarse) bound
 \begin{displaymath}
\begin{split}
 &\sum_{ \substack{(m, n) \in \Lambda\\ (mn, q) = 1}} \lambda_f(m) \lambda_g(n)W_1\Big(\frac{m}{M}\Big)W_2\Big(\frac{n}{N}\Big)\\
 & \ll_{f, g,  W_1, W_2 ,\varepsilon}  (MN)^{\varepsilon} \min\Big( \frac{(M+N)^2}{q} + \frac{(M+N)(l_1+l_2)}{q^{1/d}}, \frac{MN}{q} + \min(M, N) \Big) .
 \end{split}
 \end{displaymath}

\end{prop}	




\begin{proof}  Assume without loss of generality that $M \leq N$. 

(a) We apply Lemma \ref{thm-sieve} with   $$\lambda_1 =  |\lambda_f|, \quad \lambda_2 = |\lambda_g|,  \quad k = 2, \quad     \mathcal{S} = \Lambda \cap \mathcal{R}$$ where\begin{displaymath}
\begin{split}
&\mathcal{R} = \{(m, n) \in \Bbb{R}^2 \mid M  \leq m < 2M ,   N \leq n  < 2N\} \\
\end{split}
\end{displaymath}
and $0<\eps <1/25$ fixed, but sufficiently small . 
 Then for $(d_1d_2, q) = 1$ the sublattice $$\Lambda_{d_1, d_2} = \{(m, n) \in \Lambda : d_1 \mid m, d_2 \mid n\}$$ has the following properties: it is of the form \eqref{form}, its covolume is $q d_1d_2$ and its shortest nonzero vector is at least of length $\frac{ q^{1/d} }{2\max(l_1, l_2)}$. Thus we confirm \eqref{sieve-cond} with
$$X = \frac{MN}{q}, \quad Y = 1 + \min\Big(\min(M, N) + \frac{M+N}{q}, \frac{(M+N)(l_1+l_2)}{q^{1/d}}\Big)$$
so that $X, Y \geq 1$.

Together with \eqref{sato} we conclude from Lemma \ref{thm-sieve}  (choosing $\gamma=8\eps$ in that Lemma and remembering that $M\leq N$) that
\begin{multline}\label{prelim-error}
	\sum_{\substack{M \leq m < 2M \\ N \leq n < 2N \\ l_1 m \equiv \xi l_2 n\, (\text{mod } q)\\ (mn, q) = 1}} |\lambda_f(m) \lambda_g(n)|\ll_{\eps,\gamma} \frac{MN}{q}\log  \Big(2+\frac{MN}{q}\Big)^{\frac{16}{3\pi} -2}+(\frac{MN}q)^{1-\gamma/4+\eps}\\
+ \Big(\frac{MN}{q}\Big)^{\eps+3\gamma} \Big(1 + \min\Big(M + \frac{N}{q}, \frac{N (l_1 + l_2)}{q^{1/d}}\Big)\Big)
\end{multline}
and choose $\gamma=8\eps$ so that the second term on the righthand side of \eqref{prelim-error} is bounded by the first.

We have
$$\min\Big(M + \frac{N}{q}, \frac{N (l_1 + l_2)}{q^{1/d}}\Big)\leq \sqrt{\frac{MN (l_1 + l_2)}{q^{1/d}}}+\frac{N}q$$ so that (since $25\eps< 1$ and $MN\geq q$) the right-hand side of \eqref{prelim-error} is bounded by
$$\ll_\eps \frac{MN}{q} \log  \Big(2+\frac{MN}{q}\Big)^{\frac{16}{3\pi} -2} +     ( {MN} q)^{25\eps} \Big(\sqrt{\frac{ MN (l_1 + l_2)}{q^{1/d}}} + \frac{N}{q} \Big).$$

This last bound is sufficient except when\begin{equation}
	\label{upperboundonM}
	\frac{N}q\geq \sqrt{\frac{ MN (l_1 + l_2)}{q^{1/d}}}\Longleftrightarrow M\leq \frac{N}{q^{2-1/d}(l_1+l_2)}.
\end{equation}
In case \eqref{upperboundonM} holds,  we complement the above argument with an application of the Voronoi summation formula  on the $n$ variable. 
Writing $\mu = \bar{l}_1 l_2 \xi m$   (in particular $(\mu, q) = 1$), and detecting the congruence condition $n \equiv \mu\mods q$ with additive characters, we have
$$\sum_{n \equiv \mu \, (\text{mod }q)}\lambda_g(n) W_2\Big(\frac{n}{N}\Big)  =\frac{1}{q} \sum_{d \mid q}  \sum_{\substack{b\, (\text{mod } d)\\ (b, d) = 1}} e\Big( \frac{\mu b}{d}\Big) \sum_{n}\lambda_g(n) e\Big(- \frac{bn}{d}\Big)W_2\Big(\frac{n}{N}\Big).$$ 
Applying the Voronoi summation formula (cf. Lemma \ref{Voronoi}), this equals
	$$ \frac{1}{q} \sum_{d \mid q}   \eps(g) \frac{N}{dR^{1/2}} \sum_n \lambda_g(n) S(\mu, n\overline{R}, d) \widetilde{W}_2\Big( \frac{Nn}{d^2R}\Big).$$ 
	Using Weil's bound for  Kloosterman sum (see \eqref{Weilbound}) and estimating everything else trivially, we obtain for any $\eta>0$ that
\begin{multline}\label{unbalancedbound}
	\sum_{ \substack{ l_1 m \equiv \xi l_2 n\, (\text{{\rm mod }} q)\\ (mn, q) = 1}} \lambda_f(m) \lambda_g(n) W_1\Big(\frac{m}{M}\Big)W_2\Big(\frac{n}{N}\Big)  \\
\ll_\eta d(q)\sum_{m \asymp M} |\lambda_f(m)| \Big(\frac{1}{q} \max_{d \mid q}  \frac{N}{d} \cdot  \frac{d^2}{N} d^{1/2}\Big)^{1+\eta/2} \ll_\eta (q^{1/2} M)^{1+\eta}.
\end{multline}

Under \eqref{upperboundonM}, we have
$$q^{1/2}M=\sqrt{qMM}\leq \sqrt{\frac{MN}{q^{1-1/d}(l_1+l_2)}}\leq \frac{1}{\sqrt{q^{1-2/d}}}\sqrt{\frac{MN(l_1+l_2)}{q^{1/d}}}.$$
So since $d\geq 4$, we see that up to taking $\eta$ sufficiently small depending on $\eps$ we see that the lefthand side of \eqref{unbalancedbound} is bounded by
$$(MNq)^{25\eps}\sqrt{\frac{ MN (l_1 + l_2)}{q^{1/d}}}.$$
This concludes the proof of $(a)$.


(b) Here we estimate trivially using Lemma \ref{lattice}(b) to obtain
\begin{displaymath}
\begin{split}
\sum_{\substack{ (m, n) \in \Lambda\\ (mn, q) = 1}} \lambda_f(m) \lambda_g(n)W_1\Big(\frac{m}{M}\Big)W_2\Big(\frac{n}{N}\Big)& \ll_\eps (MN)^{\varepsilon} \sum_{ \substack{(m, n) \in \Lambda \\ m, n \ll M+N}} 1\\
& \ll_\eps (MN)^{\varepsilon} \Big(\frac{N^2}{q} + \frac{N(l_1+l_2)}{q^{1/d}}\Big).  
 \end{split}
 \end{displaymath}
Alternatively,  we can also estimate trivially
\begin{displaymath}
\begin{split}
\sum_{\substack{ (m, n) \in \Lambda \\ (mn, q) = 1}} \lambda_f(m) \lambda_g(n)W_1\Big(\frac{m}{M}\Big)W_2\Big(\frac{n}{N}\Big)& \ll_\eps (MN)^{\varepsilon} \sum_{m \ll M} \sum_{\substack{n \ll N \\ n \equiv l_1 \overline{\xi l_2}m \, (\text{mod } q)}} 1\\
& \ll_\eps (MN)^{\varepsilon} M \Big( \frac{N}{q} + 1\Big).
 \end{split}
 \end{displaymath}
This completes the proof. 
\end{proof}


\subsection{Endgame: Proof of Proposition \ref{final-error}}

With Proposition \ref{bound1} in hand, we can complete the estimation of the error term. 
Applying a smooth partition of unity to both the $n$ and $m$-sum (Lemma \ref{partition of unity}),  it suffices to estimate
\begin{displaymath}
\begin{split}
&\sum_{N = 2^{\nu}} \sum_{M= 2^{\mu}} \frac{1}{\sqrt{MN}} \Big(1 + \frac{MN}{q^2}\Big)^{-1}\Big| \sum_{\substack{ l_1 m \equiv \xi l_2 n\, (\text{{\rm mod }} q)\\ (mn, q) = 1}} \lambda_f(m) \lambda_g(n)W_1\Big(\frac{m}{M}\Big)W_2\Big(\frac{n}{N}\Big) \Big|\\
 \end{split}
 \end{displaymath}
where $\mu, \nu$ run through non-negative integers. 

Let $L = \log q/\log 2$ and let us first assume $MN \gg q$. Then by part (a) of Proposition \ref{bound1} we need to bound
\begin{displaymath}
\begin{split}
& \ll \sum_{N = 2^{\nu}} \sum_{M= 2^{\mu}}   \Big(1 + \frac{MN}{q^2}\Big)^{-1} \Big[ \frac{\sqrt{MN}}{q} \Big(\log \Big(2+\frac{MN}{q}\Big)\Big)^{\frac{16}{3\pi} -2} + 
 (MNq)^{\varepsilon}\Big(\sqrt{\frac{ l_1 + l_2}{q^{1/d}}}\Big)\Big] .
 \end{split}
 \end{displaymath}
 Clearly the second term contributes
 $$  \ll q^{ - \frac{1}{2d} + \varepsilon} \sqrt{l_1 + l_2}.$$
 We  estimate the first    by computing geometric series (and their derivatives). The first term is at most
\begin{displaymath}
\begin{split}
& \ll  \sum_{\mu + \nu  \leq 3L/2} \frac{2^{(\mu+\nu)/2}}{q} +  (\log q)^{\frac{16}{3\pi} -2} \Big(\sum_{3L/2 \leq \mu + \nu  \leq 2L} \frac{2^{(\mu+\nu)/2}}{q}  + \sum_{\mu + \nu \geq 2L} \frac{q}{2^{\mu + \nu}}\Big)\\
& \ll  q^{-1/4} (\log q)^2 +  (\log q)^{\frac{16}{3\pi} -2}\Big(L \sum_{3L/2 \leq \rho \leq 2L} \frac{2^{\rho/2}}{q} + \sum_{\rho \geq 2L} \frac{\rho q}{2^{\rho}}\Big) \ll (\log q)^{\frac{16}{3\pi} -1}.
\end{split}
\end{displaymath}

Let us now restrict to $MN \ll q$. 
By part (b) of Proposition \ref{bound1} and symmetry this contributes
\begin{displaymath}
\begin{split}
& \ll q^{\varepsilon}\underset{M \leq N, \, MN \leq q}{ \sum_{N = 2^{\nu}} \sum_{M= 2^{\mu}} } \frac{1}{\sqrt{MN}}  \min\Big( \frac{N^2}{q}, \frac{N(l_1+l_2)}{q^{1/d}}, M\Big)\\
& \ll q^{\varepsilon} \Big( \underset{M/N \geq q^{-1/d}, \, MN \leq q, \, M \leq N}{ \sum_{N = 2^{\nu}} \sum_{M= 2^{\mu}} }\frac{N^{3/2} M^{-1/2}}{q} + \frac{(N/M)^{1/2} (l_1 + l_2)}{q^{1/d}}\\&\ \hskip 6cm + \underset{M/N \leq q^{-1/d}, \, MN \leq q}{ \sum_{N = 2^{\nu}} \sum_{M= 2^{\mu}} } 
(M/N)^{1/2}\Big)\\
& \ll q^{\varepsilon} \Big( \frac{q^{\frac{1}{2} + \frac{3}{2d} }}{q} + \frac{q^{\frac{1}{2d}} (l_1 + l_2)}{q^{1/d}} + q^{-\frac{1}{2d}}\Big) \ll q^{-\frac{1}{2d} + \varepsilon} (l_1 + l_2)
 \end{split}
\end{displaymath}
since $d \geq 4$. This completes the proof.

\section{Generation of Hecke fields}\label{generation of hecke field section}

	 In this section, we show that, for $f$ satisfying \ref{non-quad}, $p$ odd and $\chi\in\Xipw$, the Hecke field $\Qq_f(\chi)$ is generated by the product $|L_f(\chi)|^2$ and a $p$-th root of unity as long as the conductor of $\chi$ is large enough.

For this we will use Theorem \ref{general l1l2 sum}
in the case $f=g$ for  $l_1=\ell^t$ a (fixed) prime power and $l_2=1$. 

In view of the shape of the main term in the asymptotic formula \eqref{startsum} and of Proposition \ref{AflLemma} we will need the following lemma.
\begin{lem}\label{v lemma}
	Let $f$ be a normalized non-CM newform. Define
\begin{align*}
	&\mcL=\mcL_{f}=\{\ell \text{ prime}: \ell\nmid pR, \ell\in 1+ p\mathbb{Z}_p^{\times} \text{ and }\lambda_f(\ell)\neq 0,\pm2\}.
\end{align*}
Then the set $\mcL$ has density  $1/p$ among the set of all primes.

For such $\ell$'s, the local parameters $\alpha_f(\ell)$ and $\beta_f(\ell)$ are distinct.
\end{lem}

\begin{proof}	
	The condition $\ell\nmid pR$ does not affect the density of $\mcL$. 
	By Dirichlet's theorem, the set of primes in $1+p\mathbb{Z}_p^{\times}$ has density $1/p$ among the set of all primes.	
	By the result of Serre \cite[p.\ 174]{serre1981quelques}, the number of prime $\ell$'s with $\lambda_{f}(\ell)=0$ or $\pm 2$ has density 0 (in fact this holds for the prime $\ell$'s such that $\lambda_{f}(\ell)$ does not equal at most finitely many numbers in $[-2,2]$). 
\end{proof} 

\begin{rem} If $(f_t)_{t\leq T}$ is any tuple of non-CM newforms,  the intersection $\bigcap_{t\leq T}\mcL_{f_t}$ has still density $1/p$. 
\end{rem}

\begin{rem}\label{remcongruence}
It is easy to see that for any $t \in \Bbb{N}$ not divisible by $p$ and any $\ell \in 1 + p\Bbb{Z}_p^{\times}$, the cogruence class $\ell^t\mods{p^h}$ 
	has order $p^{h-1}$ in $(\mathbb{Z}/p^{h}\mathbb{Z})^\times$. It follows that for all $\chi\in \Xipw$ of conductor $p^h$, the  complex number $\chi(\ell^t)$ is a primitive $p^{h-1}$-st root of unity. 
\end{rem}

\subsection{Non-vanishing of the second moment}
In the next proposition, in the case $F=\Qq_f$, we extend the Galois average in Theorem \ref{general l1l2 sum} to the trace sum over a subfield of $\Qq_f$.
	
	\begin{prop}\label{nonvan}	 Let $F=\mathbb{Q}_f$,  $c\in F\setminus\{0\}$   and $F_0$ be a subfield of $F$. Let 
	$$T_F=2\max(2, |\Hom_{\mathbb{Q}}(F,\mathbb{C})|).$$ 
	Given $\ell\in \mcL_{f}$ and  $1\leq t\leq T_F$, for $h\geq 2$ and $q=p^h$, we have   
	\begin{multline}
		\frac{\ell^{t/2}}{[F(\chi):F_0]}\Tr_{F(\chi)/F_0}(c|L_f(\chi)|^2\chi(\ell^t))
		=MT(f;q,\ell^t,F/F_0)+O((\log q)^{\frac{16}{3\pi}-1})\label{tracesumexpression}
	\end{multline}
	where
	$$MT(f;q,\ell^t,F/F_0):=\frac{1}{[F:F_0]}\sum_{\tau\in \Hom_{F_0}(F,\mathbb{C})}\frac{c^\tau}{|\Omega_{f^{\tau}}^{+}|^2}MT(f^\tau, f^{\tau};q,\ell^t,1)$$
	with the expressions $MT(f^\tau, f^{\tau};q,\ell^t,1)$ are given in Theorem \ref{general l1l2 sum} 	and the implicit constants  depend only on $f,p,\ell,c$. 
	
	If we assume \ref{non-quad} for $f$, there exists $\ell\in\mcL_{f}$ and $1\leq t\leq T$ such that   for $h\gg_{p,f,\ell,c} 1$,  one has
	$$\frac{\ell^{t/2}}{[F(\chi):F_0]}\Tr_{F(\chi)/F_0}(c|L_f(\chi)|^2\chi(\ell^t))\gg_{p,f,c,\ell}\log q .$$
	In particular $$\Tr_{F(\chi)/F_0}(c|L_f(\chi)|^2\chi(\ell^t))\not=0.$$
\end{prop}

\begin{proof}

	We have the following decomposition
	\begin{align}\label{homsetdecomposition}
		\Hom_{F_0}(F(\chi),\mathbb{C})&=\bigsqcup_{\tau\in \Hom_{F_0}(F,\mathbb{C})}\tau \text{Gal}(F(\chi)/F).
	\end{align}
	 From (\ref{shim reci}) and (\ref{homsetdecomposition}), we see that  $\Tr_{F(\chi)/F_0}(c|L_f(\chi)|^2\chi(\ell^t))$ equals
	\begin{align*}
			&\sum_{\tau\in \Hom_{F_0}(F,\mathbb{C})}\frac{c^\tau}{|\Omega_{f^{\tau}}^{+}|^2}\sum_{\sigma\in \text{Gal}(F(\chi)/F)}|L(1/2,f^{\tau}\otimes\chi^{\tau\sigma})|^2\chi^{\tau\sigma}(\ell^t).
	\end{align*}
	By Theorem \ref{general l1l2 sum}, we obtain the expression in (\ref{tracesumexpression}).

If $F_0=F$, we choose $\ell\in\mcL_{f}$ and $t=1$. By Lemma \ref{AflLemma}, we see that the main term in \eqref{tracesumexpression} is
	\begin{align*}
		\frac{c}{|\Omega_f^{+}|^2}MT(f, f;q,\ell,1)&=\frac{c}{|\Omega_f^{+}|^2}  \frac{2L^{(p)}(1, \sym^2f)}{\zeta_p(1) \zeta^{(p)}(2)}\frac{\lambda_f(\ell)}{1+\ell^{-1}}\log q+O_{f, \ell, c, p}(1)\\
		&\gg_{f,\ell,c,p} \log q.
	\end{align*}

		If $F\not=F_0$, let $t_0=[F:F_0]$ and consider the tuple
		$(f^\tau)_{\tau\in \Hom_{F_0}(F,\mathbb{C})}$ and the following tuple of vectors in $\Cc^2$: $$(C_{f^\tau,\ell},D_{f^\tau,\ell})_{\tau}:=\Big(\frac{c^\tau}{|\Omega_{f^{\tau}}^{+}|^2}\frac{2L^{(p)}(1, \sym^2f^\tau)}{\zeta_p(1) \zeta^{(p)}(2)}a_{f^\tau,\ell}(1),\frac{c^\tau}{|\Omega_{f^{\tau}}^{+}|^2}\frac{2L^{(p)}(1, \sym^2f^\tau)}{\zeta_p(1) \zeta^{(p)}(2)}b_{f^\tau,\ell}(1)\Big)_{\tau}$$ where $a_{f^\tau,\ell}(1),b_{f^\tau,\ell}(1)$ are defined in \eqref{afbfdef}.
		
		 Since $f$ has no non-trivial quadratic inner twists, by Proposition \ref{quadtwist} there exist $\ell\in \mcL_{f}$ and $t\leq T_F$ such that
		$$C(f;t,\ell,F/F_0):=\sum_{\tau\in \Hom_{F_0}(F,\mathbb{C})}\bigl(C_{f^\tau,\ell}\alpha^t_{f^\tau,\ell}(1)+D_{f^\tau,\ell}(1)\beta^t_{f^\tau,\ell}(1)\bigr)\not=0.$$ 
		Indeed  the vanishing would imply by Proposition~\ref{quadtwist}, that there exists $\tau_{1} \neq \tau_{2} \in \Hom_{F_0}(F,\mathbb{C})$ and a real valued character $\chi_d$ such that $f^{\tau_{1}} = f^{\tau_{2}} \otimes \chi_d$; this would contradict the assumption \ref{non-quad}.
		So we have
	$$
		\frac{\ell^{t/2}}{[F(\chi):F_0]}\Tr_{F(\chi)/F_0}(c|L_f(\chi)|^2\chi(\ell^t))=
		\frac{C(f;t,\ell,F/F_0)}{[F:F_0]}\log q+O_{f,\ell,c,p}(1),
	$$
	with a non-vanishing main term.
\end{proof}

Next we prove a simultaneous version of  \cite[Proposition 5.1]{sun2019generation} and  \cite[Theorem A]{luo1997determination}, which shall be used to prove Theorem \ref{lvalue generate coefficients} below.

Let $f_1,f_2$ be normalized non-CM newforms. We assume that 
\begin{enumerate}
	\item $f_1$ and $f_2$ have no nontrivial quadratic inner twists.
	\item $f_1$ is not quadradically equivalent to $f_2$ up to Galois conjugates unless $f_1=f_2$.
\end{enumerate}
\begin{thm}\label{determination L}
	 If for some $c\in\ov\Qq-\{0\}$ we have
	\begin{align}
		|L_{f_1}(\chi)|^2=c|L_{f_2}(\chi)|^2\label{determinationcondition}
	\end{align}
	for infinitely many $\chi\in \Xipw$, then under the above assumptions $f_1=f_2$.
\end{thm}
\begin{proof}
Let us set $$F_{12}:=\mathbb{Q}_{f_1,f_2}(c)$$ and $$T_{12}=2\max(2, |\Hom_{\mathbb{Q}}(F_{12},\mathbb{C})|)$$
and $\ell \in \mcL_{f_1}\cap \mcL_{f_2}$. Let $\chi\in\Xipw$ of conductor $q=p^h$ be such that $|L_{f_1}(\chi)|^2=c|L_{f_2}(\chi)|^2$.

	Applying $\sigma\in \text{Gal}(F_{12}(\chi)/F_{12})$ to (\ref{determinationcondition}), multiplying by $\chi(\ell^t)$ for some $t\geq 0$ and summing it up over $\text{Gal}(F_{12}(\chi)/F_{12})$, we have
	\begin{align*}
		\Tr_{F_{12}(\chi)/F_{12}}(|L_{f_1}(\chi)|^2\chi(\ell^t))=\Tr_{F_{12}(\chi)/F_{12}}(c|L_{f_2}(\chi)|^2\chi(\ell^t)).
	\end{align*}
	From Theorem \ref{general l1l2 sum}, we obtain 
	\begin{align*}
		\frac{MT(f_1, f_1;q,\ell^t,1)}{|\Omega_{f_1}^{+}|^2\log{q}}=\frac{c MT(f_2, f_2;q,\ell^t,1)}{|\Omega_{f_2}^{+}|^2\log{q}}+o(1).
	\end{align*}

Using Lemma \ref{AflLemma} (see also Remark \ref{remf=g} for the adjusted notations) and letting $h\rightarrow\infty$, we then obtain a system of linear equations
\begin{align*}
	D_{f_1}(a_{f_1,\ell}(1)\alpha_{f_1,\ell}^t+b_{f_1,\ell}(1)\beta_{f_1,\ell}^t)=cD_{f_2}(a_{f_2,\ell}(1)\alpha_{f_2,\ell}^t+b_{f_2,\ell}(1)\beta_{f_2,\ell}^t)
\end{align*}
for all $\ell \in \mcL_{f_1}\cap \mcL_{f_2}$ and $1\leq t\leq 4$, where 
\begin{align*}
	D_{f_i}=\frac{L^{(p)}(\sym^2f_i,1)}{|\Omega_{f_i}^{+}|^2
	},\ i=1,2.
\end{align*}

	By Proposition \ref{quadtwist}, we conclude that $f_1$ or $f_2$ has a nontrivial quadratic inner twist, or $f_1$ is quadradically equivalent to $f_2$ up to some Galois conjugates. From our assumption, it follows that $f_1=f_2$. 
\end{proof}
\subsection{The main result}
We are now able to prove that the $p$-power roots of unity and the square of the absolute value of special $L$-values generate the Hecke field. 
\begin{prop}\label{thm64}
	Assume \ref{non-quad}. Let $c\in\Qq_f-\{0\}$. Then for all but finitely many $\chi\in \Xipw$, we have
	\begin{align}\label{lvalue generate coefficients}
		\Qq_f(\chi)= \mathbb{Q}(\chi,c|L_f(\chi)|^2). 
	\end{align}
\end{prop}
\begin{proof} Set $F=\Qq_f$ and for any $\chi\in\Xipw$ let
	 $$L_\chi=\mathbb{Q}(\chi,c|L_f(\chi)|^2).$$
	Since $L_\chi\subseteq F(\chi)$, it suffices to show that for any infinite subset $S\subset \Xipw$, one has $F\subseteq L_\chi$ for infinitely many $\chi\in S$.
	
	Let $E$ be the Galois closure of $F$ over $\mathbb{Q}$ and let $E_0\subset E$ be a  subextension satisfying
	$$E_0=E\cap L_\chi$$ for infinitely many $\chi\in\Xipw$ (it exists since $E$ has finitely many subextensions). Let  	\begin{align*}
		S_0:=\{\chi\in S \mid L_\chi\cap E=E_0\}.
	\end{align*}
For any $\chi\in S_0$, the restriction to $E$ map yields an isomorphism 
$$\Gal(E(\chi)/L_\chi)\simeq \Gal(E/E_0)$$  and for $\sigma\in\Gal(E/E_0)$ we have
	\begin{align*}
			&c|L_{f}(\chi)|^2=c^\sigma L_{f}(\chi)^\sigma L_{f}(\overline{\chi})^\sigma
			=c^\sigma L_{f^\sigma}(\chi) L_{\overline{f^\sigma}}(\overline{\chi})=c^\sigma| L_{f^\sigma}(\chi)|^2.
	\end{align*}
Since $f$ is of trivial nebentypus, $\Qq_f$ is totally real (and so is $\Qq_{f^\sigma}$). 
	
	From Theorem \ref{determination L} with $(f_1, f_2) = (f, f^{\sigma})$, we conclude that $f^\sigma=f$. Therefore $F\subseteq E_0\subset L_\chi$.
\end{proof}

Next we prove that $c|L_f(\chi)|^2$ generates the $p$-power roots of unity over $\Qq(\mu_p)$.
\begin{prop}\label{l-value generate root of unity} Let $c\in \Qq_f-\{0\}$.
	Assume \ref{non-quad} and $p > 2\max(2, |\Hom_{\mathbb{Q}}(\Qq_f,\mathbb{C})|)$. Then 
	for all but finitely many $\chi\in \Xipw$, we have
	\begin{align*}
		\mathbb{Q}(\chi)\subseteq \Qq(\mu_p,c|L_f(\chi)|^2).
	\end{align*}
\end{prop}

\begin{proof}
	Let $F=\Qq_f$ and
	$$K_\chi:=\Qq(\mu_p,c|L_f(\chi)|^2)\subset F(\chi).$$
	
	Suppose $K_\chi\cap \mathbb{Q}(\chi)\neq \mathbb{Q}(\chi)$ for infinitely many $\chi$. Let $\ell\in \mcL_f$ and $1\leq t\leq T_F$. By transitivity of the trace we obtain
	\begin{align}
		\label{chi L trace}\Tr_{F(\chi)/\mathbb{Q}}(\chi(\ell^t)c|L_f(\chi)|^2)=\Tr_{K_\chi/\mathbb{Q}}\big(c|L_f(\chi)|^2\Tr_{F(\chi)/K_\chi}(\chi(\ell^t))\big).
	\end{align} for such $\chi$'s. 
	This can be rewritten as
	\begin{align*}
		[F(\chi):\mathbb{Q}(\chi)K_\chi]\Tr_{K_\chi/\mathbb{Q}}\big(c|L_f(\chi)|^2\Tr_{\mathbb{Q}(\chi)/K_\chi\cap\mathbb{Q}(\chi)}(\chi(\ell^t))\big).
	\end{align*}
	Now for any  $\chi\in\Xipw$, any $l\in \Zz$ coprime with $p$ and any 
	subfield $L\subseteq\mathbb{Q}(\chi)$ such that $\mu_p\subseteq L$, we have
	\begin{align}\label{orth}
		\Tr_{\mathbb{Q}(\chi)/L}(\chi(l))=\begin{cases}
			[\mathbb{Q}(\chi):L]\chi(l)&\text{if }\chi(l)\in L,\\
			0&\text{otherwise.}
		\end{cases}
	\end{align}
	(Indeed, the condition $\mu_p \subseteq L \subseteq \Qq(\chi)$ implies that $L = \Bbb{Q}(\mu_{p^r})$ for some $1 \leq r \leq h-1$. Then the elements of $\text{Gal}(\Bbb{Q}(\chi)/L)$ are explicitly given by $\zeta_{p^{h-1}} \mapsto \zeta_{p^{h-1}}^{n}$ for $n \equiv 1\mods{p^r}$) and the claim follows.)

	Now \eqref{orth} and our assumption $K_\chi\cap \mathbb{Q}(\chi)\neq \mathbb{Q}(\chi)$ implies that (\ref{chi L trace}) is zero for infinitely many $\chi$, since $\chi(\ell^t)$ is a primitive $p^{h-1}$-st root of unity for all $1\leq t\leq T_F$ by our assumption $p > T_F$,  cf.\ Remark \ref{remcongruence}.
	
	On the other hand, we have shown in Proposition \ref{nonvan} that there exists  at least one $\ell^t$ such that the  trace (\ref{chi L trace}) is non-vanishing for all but finitely many $\chi$. Thus we obtain a contradiction. 
\end{proof}
\begin{rem}\label{orth remark}
	The orthogonality relation (\ref{orth}) appears in  \cite[p.\ 934]{sun2019generation} except for the condition $\mu_p\subseteq L$. However, such a condition cannot be excluded. For example, if $L$ is a maximal real subfield of $\mathbb{Q}(\chi)$, we have $$\Tr_{\mathbb{Q}(\chi)/L}(\chi(l))=\chi(l)+\overline{\chi}(l)\neq 0.$$ Thus we let $K_\chi$ contain $p$-th roots of unity by choosing the base field $\mathbb{Q}(\mu_{p})$. For the same reason, the assertions of  \cite[Theorem D and Theorem 5.4]{sun2019generation} should be slightly modified and $\mu_p$ should be included. 
\end{rem}

\begin{rem} The left hand side of the following sketch shows the field diagram in the proof of Proposition \ref{l-value generate root of unity} (based on Proposition \ref{nonvan}, indicated by the dotted lines), the right hand side shows the field diagram  in the proof of Proposition \ref{thm64} (based on Theorem \ref{determination L}). 

  \xymatrixcolsep{0mm}
\begin{xy}
\xymatrix @!R{ & \Bbb{Q}_f( \chi) \ar@{-}[d] & &&& \\ & \Bbb{Q}(\chi, c|L_f(\chi)|^2)  \ar@{-}[rd]\ar@{-}[ld]&    && E(\chi) \ar@{-}[rd]\ar@{-}[ld]& \\
\Bbb{Q}(\chi)   \ar@{-}[rd]&   &\hspace{-0.5cm} \Bbb{Q}(\mu_p, c|L_f(\chi)|^2)   \ar@{-}[ld]  & E \ar@{-}[d] &&\hspace{-0.8cm}\Bbb{Q}(\chi, c|L_f(\chi)|^2) \\ & \Bbb{Q}(\chi) \cap \Bbb{Q}(\mu_p, c|L_f(\chi)|^2) \ar@{-}[d]& &\Bbb{Q}_f  \ar@{..}[lld]\ar@{..}[lluuu]& E\cap \Bbb{Q}(\chi, c|L_f(\chi)|^2)\ar@{-}[lu]\ar@{-}[ru] & \\ & \Bbb{Q} & &&&  }
\end{xy} 

\end{rem}

Combining Proposition \ref{thm64} and Proposition \ref{l-value generate root of unity}, we finally reach our main goal:
	\begin{thm}\label{main theorem later section} Let $c\in\Qq_f-\{0\}$.	Assume that $f$ satisfies \ref{non-quad} and $p > 2\max(2, |\Hom_{\mathbb{Q}}(\Qq_f,\mathbb{C})|)$. Then for all but finitely many $\chi\in \Xipw$, we have
	\begin{align*}
		\Qq_f(\chi)=\Qq(\mu_p,c|L_f(\chi)|^2).
	\end{align*}
\end{thm}

\begin{proof}Let $d_f=[\Qq_f:\Qq]$. From our assumption we have
$$[\mathbb{Q}_f(\chi):\mathbb{Q}_f]=[\mathbb{Q}(\chi):\mathbb{Q}]=\vphi(p^{h-1})=(p-1)p^{h-2},$$
and therefore  $$[\mathbb{Q}_f(\chi):\mathbb{Q}]=d_f(p-1)p^{h-2}.$$
	
	  Since $\Qq(\mu_p)/\mathbb{Q}$ is Galois and for all but finitely many  $\chi$,
	  we have
	  $$\mathbb{Q}(\mu_p,|L_f(\chi)|^2)=\mathbb{Q}_f(\chi)$$
	   by Theorem \ref{main theorem later section}, it follows that $[\Qq(\mu_p):\mathbb{Q}]=p-1$ is divisible by $d_L=[\mathbb{Q}_f(\chi):\mathbb{Q}(|L_f(\chi)|^2)]$. Now we write
	\begin{align*}
		G_1=\text{Gal}(\mathbb{Q}_f(\chi)/\mathbb{Q}(|L_f(\chi)|^2)), \quad G_2=\text{Gal}(\mathbb{Q}_f(\chi)/\mathbb{Q}_f).
	\end{align*}
	Then $G_1G_2$ is a subgroup of $\text{Aut}(\mathbb{Q}_f(\chi)/\mathbb{Q})$, so that 
	\begin{align*}
		|G_1G_2|=\frac{d_L\times (p-1)p^{h-2}}{|G_1\cap G_2|}\text{ divides }d_f(p-1)p^{h-2}.
	\end{align*}
	Since $(d_L,d_f)$=1, we see that $|G_1\cap G_2|=d_L$, thus $\mathbb{Q}_f$ is a subfield of $\mathbb{Q}_f(\chi)$ fixed by $G_1$. \end{proof}
	\begin{cor}\label{Qf corollary}
	 Suppose in addition that $(p(p-1),[\Qq_f:\Qq])=1$. Then for all but finitely many $\chi\in \Xipw$, we have
	\begin{align}
		\mathbb{Q}(|L_f(\chi)|^2)\supseteq\mathbb{Q}_f.\label{Qf inclusion last section}
	\end{align}
	 
\end{cor}
\begin{rem}
	One can replace $|L_f(\chi)|^2$ by $L_f(\chi)$ in (\ref{Qf inclusion last section}) using  \cite[Theorem \ref{sun}]{sun2019generation}.
\end{rem}

We also deduce a variant of Luo-Ramakrishnan's result \cite[Theorem C]{luo1997determination}. Let $\chi_0$ be a trivial or quadratic Dirichlet character such that $L_f(\chi_0)\not=0$ and set \begin{align*}
R_{f}(\chi,\chi_0):=\frac{|L_f(\chi)|^2}{|L_f(\chi_0)|^2}
\end{align*}
for   $\chi\in\Xipw$. 
We have
\begin{cor}
	 Assume that $f$ satisfies \ref{non-quad}. Let $p > 2\max(2, |\Hom_{\mathbb{Q}}(\Qq_f,\mathbb{C})|)$. Then for any quadratic Dirichlet character $\chi_0$ such that $|L_f(\chi_0)|^2\not=0$, we have
	\begin{align*}
		\mathbb{Q}_f(\chi)=\mathbb{Q}(\mu_p,R_{f}(\chi,\chi_0))
	\end{align*}
	for all but finitely many $\chi\in \Xipw$.
\end{cor}
\begin{proof}
	 Choose $c=|L_f(\chi_0)|^{-2}$  (which is in $\Qq_f-\{0\}$ since $\chi_0$ is trivial or quadratic) in Theorem \ref{main theorem later section}. 
\end{proof}

	\bibliography{reference}
	\bibliographystyle{plain}

\end{document}